\documentclass{amsart}

\usepackage{amsfonts,amssymb,amsmath}

\newtheorem{theorem}{Theorem} 
\newtheorem*{theorem*}{Theorem} 
\newtheorem{lemma}{Lemma} 
 
\newtheorem{proposition}{Proposition}

\title{Automorphisms of a Free Centre-by-Centre-by-Metabelian Group of Rank 3} 
\author{C. E. Kofinas} 
\address{University of the Aegean, Department of Mathematics, Karlovassi GR-83200, Samos, Greece}
\email{kkofinas@aegean.gr}

\begin{document}

\begin{abstract} 
Let $F_{3}$ be the free group of rank $3$ and let $G_{3} = F_{3}/[F_{3}^{\prime\prime}, F_{3}, F_{3}]$, that is, $G_{3}$ is a free centre-by-centre-by-metabelian group of rank $3$. We show that ${\rm Aut}(G_{3})$ contains a proper finitely generated subgroup that is dense with respect to the formal power series topology.

\bigskip


\bigskip

\end{abstract}

\maketitle

\section{Introduction}\label{sec1}

Let $F_{3}$ be the free group of rank $3$ freely generated by the set $\{f_{1}, f_{2}, f_{3}\}$. We write $G_{3} = F_{3}/[F_{3}^{\prime\prime}, F_{3}, F_{3}]$ and $x_{i} = f_{i}[F_{3}^{\prime\prime}, F_{3}, F_{3}]$ for $i = 1, 2, 3$. Thus, $G_{3}$ is a free centre-by-centre-by-metabelian group of rank $3$, freely generated by the set $\{x_{1}, x_{2}, x_{3}\}$. The natural mapping from $F_{3}$ onto $G_{3}$ induces a group homomorphism from ${\rm Aut}(F_{3})$ into ${\rm Aut}(G_{3})$. The image of this homomorphism is denoted by $T_{3}$, and the automorphisms of $G_{3}$ that belong to $T_{3}$ are called tame; otherwise they are called non-tame (or wild). 

It was shown in \cite{hur} that the intersection of the lower central series of $G_{3}$ is trivial. Hence, we may consider the topology on $G_{3}$ corresponding to the lower central series $G_{3} \supseteq \gamma_{2}(G_{3}) \supseteq \gamma_{3}(G_{3}) \supseteq \cdots$. For $c \geq 2$, the natural mapping from $G_{3}$ onto $G_{3}/\gamma_{c}(G_{3})$ induces a group homomorphism from ${\rm Aut}(G_{3})$ into ${\rm Aut}(G_{3}/\gamma_{c}(G_{3}))$. Write ${\rm I}_{c}{\rm A}_{G_{3}}$ for the kernel of this homomorphism. By \cite{and}, $\bigcap_{i \geq 2}{\rm I}_{i}{\rm A}_{G_{3}} = \{1\}$. Therefore, the topology corresponding to the lower central series induces a topology on ${\rm Aut}(G_{3})$ that corresponds to the series
\begin{equation*}
{\rm Aut}(G_{3}) \supset {\rm I}_{2}{\rm A}_{G_{3}} \supset {\rm I}_{3}{\rm A}_{G_{3}} \supset \ldots.
\end{equation*}
This topology is called the formal power series topology.

The motivation for studying ${\rm Aut}(G_{3})$ comes from recent progress on automorphism groups of relatively free groups. For a free centre-by-metabelian group of rank $3$, it was shown in \cite{kof1} that its automorphism group contains a finitely generated subgroup that is dense with respect to the formal power series topology. In rank $2$, it remains open whether a finitely generated dense subgroup exists. In rank $n \geq 4$, the automorphism group of a free centre-by-centre-by-metabelian group of rank $n$ is finitely generated (see \cite{kof2}) and hence the problem of finding a finitely generated dense subgroup is trivial. It is therefore natural to ask whether a finitely generated dense subgroup exists in ${\rm Aut}(G_{3})$. Our aim is to extend the corresponding result from \cite{kof1} for a free centre-by-metabelian group of rank $3$ and prove the following result.

\begin{theorem}\label{the}
Let $G_{3} = F_{3}/[F_{3}^{\prime\prime}, F_{3}, F_{3}]$ be a free centre-by-centre-by-metabelian group of rank $3$ freely generated by the set $\{x_{1}, x_{2}, x_{3}\}$. Let ${\rm Aut}(G_{3})$ be the automorphism group of $G_{3}$ and let $T_{3}$ be the group of tame automorphisms of $G_{3}$. Then, the subgroup generated by $T_{3}$ and the non-tame automorphisms $\mu$, $w$ and $\alpha$ defined by
\begin{equation*}
\begin{split}
\mu(x_{1}) &= x_{1}[x^{-1}_{1},[x_{1},[x_{2},x_{3}]]], \mu(x_{i}) = x_{i} \text{ for } i=2, 3,\\
w(x_{1}) &= x_{1}[[x_{1},x_{2}],[x_{1},x_{3}]], w(x_{i}) = x_{i} \text{ for } i=2, 3,\\
\alpha(x_{1}) &= x_{1}[[x_{1},x_{2}],[x_{1},x_{3}],x_{1}], \alpha(x_{i}) = x_{i} \text{ for } i=2, 3,
\end{split}
\end{equation*}
is a proper subgroup of ${\rm Aut}(G_{3})$ that is dense in ${\rm Aut}(G_{3})$ with respect to the formal power series topology.
\end{theorem}

Note that, by \cite{che}, the automorphism $\mu$ is non-tame, and by \cite{gule}, the automorphisms $w$ and $\alpha$ are non-tame.

It was proved in \cite{nie2} that ${\rm Aut}(F_{3})$ is finitely generated. Therefore, $T_{3}$ is also finitely generated. Thus, Theorem \ref{the} shows that ${\rm Aut}(G_{3})$ contains a dense subgroup generated by an explicit finite set.

\section{Preliminaries and Background}\label{sec2}

\subsection{Notation}

Let $G$ be a group. For $a, b \in G$, we write $a^{b} = b^{-1}ab$ and $[a, b] = a^{-1}b^{-1}ab$ for the commutator of $a$ and $b$. We adopt the left-normed convention for commutators: for $g_{1}, \ldots, g_{c} \in G$, with $c \geq 3$, we write $[g_{1}, \ldots, g_{c}] = [[g_{1}, \ldots, g_{c-1}], g_{c}]$. For subgroups $A$ and $B$ of $G$, we denote by $[A, B]$ the subgroup of $G$ generated by all commutators $[a, b]$, where $a \in A$ and $b \in B$. We adopt the left-normed convention for subgroups as well, so that for subgroups $A$, $B$ and $C$ of $G$, we write $[A, B, C] = [[A, B], C]$. For a subset $X$ of $G$, $\langle X \rangle$ denotes the subgroup of $G$ generated by $X$. For a subgroup $H$ of $G$, $\langle X \rangle^{H}$ denotes the subgroup of $G$ generated by all the elements $y^{-1}xy$, where $y \in H$ and $x \in X$.  We write $G^{\prime}$ for the derived subgroup of $G$, that is, $G^{\prime} = [G, G]$, and $G^{\prime\prime} = (G^{\prime})^{\prime}$. For a positive integer $c$, let $\gamma_{c}(G)$ be the $c$-th term of the lower central series of $G$. Note that $\gamma_{2}(G) = G^{\prime}$. We say that $G$ is residually nilpotent if $\bigcap_{i \geq 1}\gamma_{i}(G) = \{1\}$. 

Let ${\rm Aut}(G)$ be the automorphism group of $G$. For $c \geq 2$, the natural group epimorphism from $G$ onto $G/\gamma_{c}(G)$ induces a group homomorphism from ${\rm Aut}(G)$ into ${\rm Aut}(G/\gamma_{c}(G))$. We denote the kernel of this homomorphism by ${\rm I}_{c}{\rm A}_{G}$. The elements of ${\rm I}_{2}{\rm A}_{G}$, that is, the automorphisms of $G$ that induce the identity mapping on the abelianization of $G$, are called IA-automorphisms of $G$. It was proved in \cite[Theorem 1]{and} that the series
${\rm I}_{2}{\rm A}_{G} \supset {\rm I}_{3}{\rm A}_{G} \supset \ldots$ is central. Furthermore, if $G$ is residually nilpotent, then by \cite[Theorem 2]{and}, $\bigcap_{i \geq 2}{\rm I}_{i}{\rm A}_{G} = \{1\}$.

\subsection{Automorphisms of Relatively Free Groups}

Let $F_{n}$ be the free group of finite rank $n$, with $n \geq 2$, freely generated by the set $\{f_{1}, \ldots, f_{n}\}$.  Let $V$ be a variety of groups and let $V(F_{n})$ be the verbal subgroup of $F_{n}$ corresponding to $V$. Thus, the group $F_{n}(V) = F_{n}/V(F_{n})$ is a relatively free group of rank $n$ in $V$. For background on varieties of groups and relatively free groups, see \cite{hane}. The natural mapping from $F_{n}$ onto $F_{n}(V)$ induces a group homomorphism from ${\rm Aut}(F_{n})$ into ${\rm Aut}(F_{n}(V))$. The image of this homomorphism is denoted by $T_{F_{n}(V)}$, and the automorphisms of $F_{n}(V)$ which belong to $T_{F_{n}(V)}$ are called tame; otherwise they are called non-tame (or wild).

It is well known that ${\rm Aut}(F_{n})$ is finitely generated and a generating set has been given in \cite{nie2}. For $n = 3$, the Nielsen automorphisms $k_1$, $\pi_{1,2}$, $\sigma_{1,2}$ and $\tau$, defined by 
\begin{align*}
k_{1}(f_{1}) &= f^{-1}_{1},  k_{1}(f_{2}) = f_{2},  k_{1}(f_{3}) = f_{3}, \\
\pi_{1,2}(f_{1}) &= f_{1}f_{2}, \pi_{1,2}(f_{2}) = f_{2}, \pi_{1,2}(f_{3}) = f_{3}, \\
\sigma_{1,2}(f_{1}) &= f_{2},  \sigma_{1,2}(f_{2}) = f_{1},  \sigma_{1,2}(f_{3}) = f_{3}, \\
\tau(f_{1}) &= f_{2},  \tau(f_{2}) = f_{3},  \tau(f_{3}) = f_{1},
\end{align*}
form a generating set of ${\rm Aut}(F_{3})$. Since the tame automorphisms of $G_3$ are induced by automorphisms of $F_{3}$, the corresponding automorphisms of $G_3$, which we also denote by $k_{1}, \pi_{1,2}, \sigma_{1,2}, \tau$, generate $T_{3}$. For a generating set of $T_{3}$ consisting of only two elements, see \cite{new}.

Let $n$ be a positive integer, with $n \geq 2$. We write $M_{n} = F_{n}/F_{n}^{\prime\prime}$ and $z_{i} = f_{i}F_{n}^{\prime\prime}$ for $i = 1, \ldots, n$. Thus, $M_{n}$ is a free metabelian group of rank $n$ freely generated by the set $\{z_{1}, \ldots, z_{n}\}$. We write $C_{n} = F_{n}/[F_{n}^{\prime\prime}, F_{n}]$ and $y_{i} = f_{i}[F_{n}^{\prime\prime}, F_{n}]$ for $i = 1, \ldots, n$. Thus, $C_{n}$ is a free centre-by-metabelian group of rank $n$ freely generated by the set $\{y_{1}, \ldots, y_{n}\}$. Furthermore, we write $G_{n} = F_{n}/[F_{n}^{\prime\prime}, F_{n}, F_{n}]$ and $x_{i} = f_{i}[F_{n}^{\prime\prime}, F_{n}, F_{n}]$ for $i = 1, \ldots, n$. Thus, $G_{n}$ is a free centre-by-centre-by-metabelian group of rank $n$ freely generated by the set $\{x_{1}, \ldots, x_{n}\}$. Observe that $[G_{n}^{\prime\prime}, G_{n}] = [F_{n}^{\prime\prime}, F_{n}]/[F_{n}^{\prime\prime}, F_{n}, F_{n}]$. Therefore, $G_{n}/[G_{n}^{\prime\prime}, G_{n}] \cong C_{n}$. Moreover, by \cite[Theorem 6.1]{ngule}, $[G_{n}^{\prime\prime}, G_{n}]$ is the centre of $G_{n}$.

Note that for $n \neq 3$, all automorphisms of $M_{n}$ are tame (see \cite{bamu2}, \cite{rom}). For $n = 3$, ${\rm Aut}(M_{3})$ is not finitely generated (see \cite{bamu}). In particular, ${\rm Aut}(M_{3})$ contains non-tame automorphisms.

For the proof of the following result, see \cite[Lemma 3.1]{brgu}.

\begin{proposition}\label{propo1}
For integers $n, c \geq 2$, let $V(F_{n})$ be a fully invariant subgroup of $F_{n}$ such that $\gamma_{c+1}(F_{n}^{\prime}) \subseteq V(F_{n}) \subseteq F_{n}^{\prime\prime}$, and let $W_{n} = F_{n}/V(F_{n})$. Then, an endomorphism $\phi$ of $F_{n}$ induces an automorphism of $W_{n}$ if and only if $\phi$ induces an automorphism of the free metabelian group $M_{n}$.
\end{proposition}

Let $n \geq 2$.  Since $G_{n}/[G_{n}^{\prime\prime}, G_{n}] \cong C_{n}$, the natural mapping from $G_{n}$ onto $C_{n}$ induces a group homomorphism from ${\rm Aut}(G_{n})$ into ${\rm Aut}(C_{n})$. Similarly, since $C_{n}/C_{n}^{\prime\prime} \cong M_{n}$, the natural mapping from $C_{n}$ onto $M_{n}$ induces a group homomorphism from ${\rm Aut}(C_{n})$ into ${\rm Aut}(M_{n})$. By Proposition \ref{propo1}, we get the following result.

\begin{lemma}\label{leem1}
For an integer $n \geq 2$, let $M_{n}$ be a free metabelian group of rank $n$, let $C_{n}$ be a free centre-by-metabelian group of rank $n$ and let $G_{n}$ be a free centre-by-centre-by-metabelian group of rank $n$. Then:
\begin{enumerate}
\item An endomorphism of $G_{n}$ is an automorphism if and only if it induces an automorphism of $M_{n}$.

\item The group homomorphism from ${\rm Aut}(G_{n})$ into ${\rm Aut}(C_{n})$, induced from the natural mapping from $G_{n}$ onto $C_{n}$ is surjective. 

\item The group homomorphism from ${\rm Aut}(C_{n})$ into ${\rm Aut}(M_{n})$, induced from the natural mapping from $C_{n}$ onto $M_{n}$ is surjective.
\end{enumerate}
\end{lemma}

The group ${\rm Aut}(C_{n})$ has been studied in \cite{sto}. For $2 \leq n \leq 3$, it was proved in \cite[Theorem]{sto} that ${\rm Aut}(C_{n})$ is not finitely generated. Furthermore, a minimal (infinite) set of generators of ${\rm Aut}(C_{2})$ has been given in \cite{kof1}. For $n \geq 4$, it was proved in \cite[Theorem]{sto} that ${\rm Aut}(C_{n})$ is generated by the tame automorphisms and one more non-tame automorphism (the non-tameness of this automorphism follows from \cite{gule}). Note that this result has been generalized in \cite{pap}.
 
Let $2 \leq n \leq 3$. Since, by Lemma \ref{leem1} (2), there is a group epimorphism from ${\rm Aut}(G_{n})$ onto ${\rm Aut}(C_{n})$, and ${\rm Aut}(C_{n})$ is not finitely generated, it follows that ${\rm Aut}(G_{n})$ is not finitely generated. For $n \geq 4$, using the work of \cite{pap}, it was proved in \cite[Proposition 4]{kof2} that ${\rm Aut}(G_{n})$ is generated by the tame automorphisms and one more non-tame automorphism.

\subsection{The Formal Power Series Topology}

For our terminology and notation, we follow \cite{dp} (see also \cite{bd}). Let $G$ be a relatively free residually nilpotent group of finite rank $n$, with $n \geq 2$. For an integer $k \geq 2$ and for a subgroup $N$ of ${\rm Aut}(G)$, we write ${\rm I}_{k}N = N \cap {\rm I}_{k}{\rm A}_{G}$ (that is, ${\rm I}_{k}N$ is the set of elements of $N$ that induce the identity mapping on $G/\gamma_{k}(G)$). Note that ${\rm I}_{k}N$ is a normal subgroup of ${\rm I}_{2}N$. The following result is a special case of \cite[Proposition 2.1]{dp}.

\begin{proposition}\label{propo2}
For a positive integer $n \geq 2$, let $G$ be a relatively free residually nilpotent group of finite rank $n$. Let $N$ be a subgroup of ${\rm Aut}(G)$ such that $T_{G} \subseteq N$. Then, $N$ is dense in ${\rm Aut}(G)$ with respect to the formal power series topology, if and only if ${\rm I}_{k}{\rm A}_{G} = ({\rm I}_{k}N){\rm I}_{k+1}{\rm A}_{G}$ for all $k \geq 2$.
\end{proposition}

Note that $M_{3}$ is residually nilpotent by \cite{gru}. Hurley \cite[Proposition 7.11 and its Corollary]{hur} claimed that the quotient $F_{n}^{\prime\prime}/[F_{n}^{\prime\prime}, F_{n}]$ is free abelian for all $n \geq 2$. Using this, he constructed a faithful embedding of $C_{n}$ into the unit group of a suitable graded power-series ring over $\mathbb{Z}$. From this embedding, he deduced that $C_{n}$ is residually nilpotent (in fact, residually torsion-free nilpotent). For $n\geq 4$, this argument fails. Gupta \cite{gup} exhibited a non-trivial fully invariant elementary abelian $2$-subgroup in the centre of $C_{n}$. This torsion makes Hurley's embedding non-faithful. However, as explicitly observed in \cite[Remark 2]{gup}, $F_{3}^{\prime\prime}/[F_{3}^{\prime\prime}, F_{3}]$ is free abelian, $C_{3}$ admits Hurley's embedding and is residually nilpotent. The construction of $G_{n}$ is inductive and uses only the fact that $F_{n}^{\prime\prime}/[F_{n}^{\prime\prime}, F_{n}]$ is free abelian (see \cite[Lemma 8.1, Proposition 8.2 and Theorem 8.3]{hur}). Although this fails for $n \geq 4$, it is true for $n = 3$. Hence the faithful embedding of $G_{3}$ (\cite[Theorem 8.3]{hur}) and the resulting residual nilpotency of $G_{3}$ (\cite[Theorem 8.4]{hur}) remain valid. Therefore, ${\rm Aut}(M_{3})$, ${\rm Aut}(C_{3})$ and ${\rm Aut}(G_{3})$ all admit the formal power series topology induced by their respective lower central series.

\section{A Dense Subgroup of ${\rm Aut}(G_{3})$}

Throughout the paper, the term ``dense'' refers to density with respect to the formal power series topology. It follows from \cite{che} that there is a subgroup of ${\rm Aut}(M_{3})$ generated by the tame automorphisms and one more non-tame automorphism of $M_{3}$ that is dense in ${\rm Aut}(M_{3})$. Using this result for ${\rm Aut}(M_{3})$, the following result was proved in \cite[Theorem 3]{kof1}.

\begin{proposition}\label{propo3}
Let $C_{3} = F_{3}/[F_{3}^{\prime\prime}, F_{3}]$ be a free centre-by-metabelian group of rank $3$ freely generated by the set $\{y_{1}, y_{2}, y_{3}\}$. Then, the subgroup of ${\rm Aut}(C_{3})$ generated by $T_{C_{3}}$ and the automorphisms $\widehat\mu$ and $\widehat w$ defined by 
\begin{equation*}
\begin{split}
\widehat\mu(y_{1}) &= y_{1}[y^{-1}_{1}, [y_{1}, [y_{2}, y_{3}]]], \widehat\mu(y_{i}) = y_{i}, \text{ for } i = 2, 3, \\
\widehat w(y_{1}) &= y_{1}[[y_{1}, y_{2}], [y_{1}, y_{3}]], \widehat w(y_{i}) = y_{i}, \text{ for } i = 2, 3,
\end{split}
\end{equation*}
is dense in ${\rm Aut}(C_{3})$, with respect to the formal power series topology.
\end{proposition}

Using Proposition \ref{propo3}, we may find a subgroup of ${\rm Aut}(G_{3})$ that is dense in ${\rm Aut}(G_{3})$. From now on, we write $\phi_{3}$ for the group epimorphism from ${\rm Aut}(G_{3})$ onto ${\rm Aut}(C_{3})$ induced by the natural mapping from $G_{3}$ onto $C_{3}$ and $H_{3}$ for the kernel of $\phi_{3}$.

\begin{proposition}\label{propo4}
Let $G_{3} = F_{3}/[F_{3}^{\prime\prime}, F_{3}, F_{3}]$ be a free centre-by-centre-by-metabelian group of rank $3$ freely generated by the set $\{x_{1}, x_{2}, x_{3}\}$. Let $\mu$ and $w$ be the automorphisms of $G_{3}$ defined by 
\begin{equation*}
\begin{split}
\mu(x_{1}) &= x_{1}[x^{-1}_{1}, [x_{1}, [x_{2}, x_{3}]]], \mu(x_{i}) = x_{i}, \text{ for } i = 2, 3, \\
w(x_{1}) &= x_{1}[[x_{1}, x_{2}], [x_{1}, x_{3}]],  w(x_{i}) = x_{i}, \text{ for } i = 2, 3.
\end{split}
\end{equation*}
Then, the subgroup $\langle T_{3}, w, \mu \rangle H_{3}$ of ${\rm Aut}(G_{3})$ is dense in ${\rm Aut}(G_{3})$, with respect to the formal power series topology.
\end{proposition}

\begin{proof}
Throughout the proof, we write $N_{1} = \langle T_{3}, \mu, w \rangle$ and $N_{2} = N_{1}H_{3}$. By Proposition \ref{propo2}, it suffices to show that ${\rm I}_{k}{\rm A}_{G_{3}} = ({\rm I}_{k}N_{2}){\rm I}_{k+1}{\rm A}_{G_{3}}$ for all $k \geq 2$. Fix $k \geq 2$. It suffices to show that ${\rm I}_{k}{\rm A}_{G_{3}} \subseteq ({\rm I}_{k}N_{2}){\rm I}_{k+1}{\rm A}_{G_{3}}$. Write $\widehat N_{1} = \langle T_{C_{3}}, \widehat\mu, \widehat w \rangle$. By Proposition \ref{propo2}, we have ${\rm I}_{k}{\rm A}_{C_{3}} = ({\rm I}_{k}\widehat N_{1}){\rm I}_{k+1}{\rm A}_{C_{3}}$. We observe that $\phi_{3}(N_{1}) = \widehat N_{1}$. Furthermore, by Lemma \ref{leem1} (2), $\phi_{3}$ is surjective and hence $\phi_{3}({\rm I}_{k}{\rm A}_{G_{3}}) = {\rm I}_{k}{\rm A}_{C_{3}}$. Therefore,
\begin{equation*}
\phi_{3}({\rm I}_{k}{\rm A}_{G_{3}}) = {\rm I}_{k}{\rm A}_{C_{3}} = ({\rm I}_{k}\widehat N_{1}){\rm I}_{k+1}{\rm A}_{C_{3}}= \phi_{3}(({\rm I}_{k}N_{1}){\rm I}_{k+1}{\rm A}_{G_{3}}).
\end{equation*}
 Let $\theta \in {\rm I}_{k}{\rm A}_{G_{3}}$. By the above equation, we may write $\theta = g_{1}g_{2}h$, where $g_{1} \in {\rm I}_{k}N_{1}$, $g_{2} \in {\rm I}_{k+1}{\rm A}_{G_{3}}$ and $h \in H_{3}$. Therefore, $\theta = g_{1}hg_{2}[g_{2}, h]$. Since $N_{1} \leq N_{2}$, $g_{1} \in {\rm I}_{k}N_{2}$. Note that $h \in H_{3} \cap {\rm I}_{k}{\rm A}_{G_{3}} = {\rm I}_{k}H_{3}$. Since $H_{3} \leq N_{2}$, it follows that $h \in {\rm I}_{k}N_{2}$. Furthermore, since ${\rm I}_{k+1}{\rm A}_{G_{3}}$ is a normal subgroup of ${\rm Aut}(G_{3})$, $g_{2}[g_{2}, h] \in {\rm I}_{k+1}{\rm A}_{G_{3}}$. Hence, $\theta \in {\rm I}_{k}N_{2}{\rm I}_{k+1}{\rm A}_{G_{3}}$, and this proves the inclusion.
\end{proof}

\section{Structure of the Kernel $H_{3}$}\label{sec3}

\subsection{Identical Relations in $G_{3}$}

By the standard group commutator identities, we have 
\begin{equation*}
\begin{split}
[ab, c] &= [a, c][a, c, b][b, c], [a, bc] = [a, c][a, b][a, b, c], \\
[a^{-1}, b] &= [a, b]^{-1}[[a, b]^{-1}, a^{-1}], [a, b^{-1}] = [a, b]^{-1}[[a, b]^{-1}, b^{-1}].
\end{split}
\end{equation*}
Moreover, by the Three Subgroups Lemma, for any group $G$ we have $[G^{\prime\prime}, G^{\prime}] \subseteq [G^{\prime\prime}, G, G]$. Therefore, it is elementary to prove the following results, which will be used repeatedly.

\begin{lemma}\label{leem2}
Let $G$ be a centre-by-centre-by-metabelian group, that is, $[G^{\prime\prime}, G, G] = \{1\}$, and let $x, y, z, t \in G$, $u, v, w \in G^{\prime}$ and $r, s \in G^{\prime\prime}$. Then: 
\begin{enumerate}
\item $[u, v, w] = 1$.

\item $[u^{-1}, v] = [u, v]^{-1}$.

\item $[uv, w] = [u, w][v,w]$ and $[u, vw] = [u, v][u, w]$.

\item $[u^{v}, w] = [u, w]$.

\item If $x \equiv y \pmod {G^{\prime}}$, then $[u^{x}, v] = [u^{y}, v]$. In particular, $[u^{xy}, v] = [u^{yx}, v]$.

\item $[rs, x] = [r, x][s, x]$.

\item $[r, xy] = [r, x][r, y]$.

\item $[u, v, x^{-1}] = [u^{-1}, v, x] = [u, v^{-1}, x] = [[u, v]^{-1}, x] = [u, v, x]^{-1}$.

\item $[u^{x}, v^{y}, z] = [u^{xy^{-1}}, v, z]$.

\end{enumerate}
\end{lemma}

\begin{lemma}\label{extrlem1}
\begin{enumerate}
\item Let $\theta$ be an automorphism of $G_{3}$ that induces the identity map on $G_{3}/G_{3}^{\prime\prime}$. Then, $\theta(u) = u$ for all $u \in G_{3}^{\prime\prime}$.

\item Let $\theta_{1}$ and $\theta_{2}$ be automorphisms of $G_{3}$ that induce the identity map on $G_{3}/G_{3}^{\prime\prime}$. Then, $\theta_{1}\theta_{2} = \theta_{2}\theta_{1}$.
\end{enumerate}
\end{lemma}

\subsection{The Kernel $H_{3}$}

Recall that $H_{3}$ is the kernel of the group epimorphism $\phi_{3}$ from ${\rm Aut}(G_{3})$ onto ${\rm Aut}(C_{3})$. Thus, $H_{3}$ consists of all automorphisms of $G_{3}$ that induce the identity mapping on $G_{3}/[G_{3}^{\prime\prime}, G_{3}]$. Equivalently, $\theta \in H_{3}$ if and only if $\theta(x_{i}) = x_{i}c_{i}$, $i = 1, 2, 3$, for some $c_{1}, c_{2}, c_{3} \in [G_{3}^{\prime\prime}, G_{3}]$. By Lemma \ref{extrlem1}, $H_{3}$ is an abelian group that acts trivially on $G_{3}^{\prime\prime}$. This fact will be used repeatedly without further notice.

By Proposition \ref{propo4}, the subgroup $\langle T_{3}, \mu, w\rangle H_{3}$ of ${\rm Aut}(G_{3})$ is dense in ${\rm Aut}(G_{3})$. Therefore, studying the kernel $H_{3}$ is crucial for proving Theorem \ref{the}. The study of $H_{3}$ is organized as follows. In this section, we study $H_{3}$ as a group on which $T_{3}$ acts as a group of operators by conjugation and give a generating set of $H_{3}^{T_{3}}$. In Section \ref{sec4}, we find a connection between elements in $H_{3}$ and automorphisms of $C_{3}$. Using this connection, we transfer computational results on ${\rm Aut}(C_{3})$ (proved in \cite{kof1}) to $H_{3}$. In Section \ref{sec5}, by a reductive method we show that $H_{3}$ is the normal closure under $T_{3}$ of two automorphisms. 

\subsection{Generating Sets of the Kernel $H_{3}$}\label{sec3}

Throughout Sections 4--7, we use the notation
\begin{equation*}
A = [x_{1}, x_{2}], B = [x_{1}, x_{3}], C = [x_{2}, x_{3}].
\end{equation*}
For our purposes, for $i, j, \ell \in \{1, 2, 3\}$, we define the following tame automorphisms of $G_{3}$, which we will use repeatedly:
\begin{itemize}
\item[] $k_{i}(x_{i}) = x^{-1}_{i}$, $k_{i}(x_{\ell}) = x_{\ell} \text{ for } \ell \neq i$.

\item[] $\sigma_{i, j}(x_{i}) = x_{j}$, $\sigma_{i, j}(x_{j}) = x_{i}$, $\sigma_{i, j}(x_{\ell}) = x_{\ell} \text{ for } \ell \neq i, j$.

\item[] $\pi_{i, j}(x_{i}) = x_{i}x_{j}$, $\pi_{i, j}(x_{\ell}) = x_{\ell} \text{ for } \ell \neq i$.
\end{itemize}

With $g \equiv x_{1}^{m_{1}}x_{2}^{m_{2}}x_{3}^{m_{3}} \pmod {G^{\prime}_{3}}$, where $m_{1}, m_{2}, m_{3} \in \mathbb{Z}$, we define the following endomorphisms of $G_{3}$:
\begin{enumerate}
\item[] $\alpha_{i, (m_{1}, m_{2}, m_{3})}(x_{1}) = x_{1}[[x_{1}, x_{2}]^{g}, [x_{1}, x_{3}], x_{i}], \quad \alpha_{i, (m_{1}, m_{2}, m_{3})}(x_{j}) = x_{j} \text{ for } j = 2, 3$,

\item[] $\beta_{i, (m_{1}, m_{2}, m_{3})}(x_{1}) = x_{1}[[x_{1}, x_{2}]^{g}, [x_{2}, x_{3}], x_{i}], \quad \beta_{i, (m_{1}, m_{2}, m_{3})}(x_{j}) = x_{j} \text{ for } j = 2, 3$,

\item[] $f_{1, i, (m_{1}, m_{2}, m_{3})}(x_{1}) = x_{1}[[x_{1}, x_{2}]^{g}, [x_{1}, x_{2}], x_{i}], \quad f_{1, i, (m_{1}, m_{2}, m_{3})}(x_{j}) = x_{j} \text{ for } j = 2, 3$,

\item[] $f_{2, i, (m_{1}, m_{2}, m_{3})}(x_{1}) = x_{1}[[x_{1}, x_{3}]^{g}, [x_{1}, x_{3}], x_{i}], \quad f_{2, i, (m_{1}, m_{2}, m_{3})}(x_{j}) = x_{j} \text{ for } j = 2, 3$,

\item[] $f_{3, i, (m_{1}, m_{2}, m_{3})}(x_{1}) = x_{1}[[x_{1}, x_{3}]^{g}, [x_{2}, x_{3}], x_{i}], \quad f_{3, i, (m_{1}, m_{2}, m_{3})}(x_{j}) = x_{j} \text{ for } j = 2, 3$,

\item[] $f_{4, i, (m_{1}, m_{2}, m_{3})}(x_{1}) = x_{1}[[x_{2}, x_{3}]^{g}, [x_{2}, x_{3}], x_{i}], \quad f_{4, i, (m_{1}, m_{2}, m_{3})}(x_{j}) = x_{j} \text{ for } j = 2, 3$.
\end{enumerate}

Note that, by Lemma \ref{leem1} (1), all the above endomorphisms are automorphisms of $G_{3}$.

\begin{lemma}\label{leem3}
$H_{3}$ is the normal closure under $T_{3}$ of all the automorphisms $\alpha_{i, (m_{1}, m_{2}, m_{3})}$, $\beta_{i, (m_{1}, m_{2}, m_{3})}$ and $f_{j, i, (m_{1}, m_{2}, m_{3})}$, with $i = 1, 2, 3$ and $j = 1, \ldots, 4$. 
\end{lemma}

\begin{proof}
Recall that an element $\theta \in H_{3}$ is defined by $\theta(x_{i}) = x_{i}u_{i}$, where $u_{i} \in [G_{3}^{\prime\prime}, G_{3}]$, $i = 1, 2, 3$. Observe that $\theta = \theta_{1}\theta_{2}^{\sigma_{1,2}}\theta_{3}^{\sigma_{1,3}}$, where $\theta_{1}$, $\theta_{2}$ and $\theta_{3}$ are the elements of $H_{3}$ defined by $\theta_{1}(x_{1}) = x_{1}u_{1}$, $\theta_{2}(x_{1}) = x_{1}\sigma_{1,2}(u_{2})$, $\theta_{3}(x_{1}) = x_{1}\sigma_{1,3}(u_{3})$ and $\theta_{1}(x_{j}) = \theta_{2}(x_{j}) = \theta_{3}(x_{j}) = x_{j}$ for $j = 2, 3$. Therefore, $H_{3}$ is the normal closure under $T_{3}$ of all automorphisms of the form 
\begin{equation*}
x_{1} \mapsto x_{1}u,\quad  x_{i} \mapsto x_{i} \text{ for } i = 2, 3,
\end{equation*}
where $u \in [G_{3}^{\prime\prime}, G_{3}]$. Note that every element of $G_{3}^{\prime\prime}$ can be written as a product of elements of the form $[[w_{1}, w_{2}]^{g_{1}}, [w_{3}, w_{4}]^{g_{2}}]$, where $w_{1}, \ldots, w_{4} \in \{x_{1}, x_{2}, x_{3}\}$. Hence, by Lemma \ref{leem2} (1), Lemma \ref{leem2} (6) and Lemma \ref{leem2} (7), we deduce that every element of $[G_{3}^{\prime\prime}, G_{3}]$ may be written as a product of elements of the form $[[w_{1}, w_{2}]^{g_{1}}, [w_{3}, w_{4}]^{g_{2}}, w_{5}]$, where $w_{1}, \ldots, w_{5} \in \{x_{1}, x_{2}, x_{3}\}$ and $g_{1}, g_{2} \in G_{3}$. Since, by Lemma \ref{leem2} (9), 
\begin{equation*}
[[w_{1}, w_{2}]^{g_{1}}, [w_{3}, w_{4}]^{g_{2}}, w_{5}] = [[w_{1}, w_{2}]^{g_{1}g^{-1}_{2}}, [w_{3}, w_{4}], w_{5}], 
\end{equation*}
it follows that $H_{3}$ is the normal closure under $T_{3}$ of all automorphisms of the form 
\begin{equation*}
x_{1} \mapsto x_{1}[[w_{1}, w_{2}]^{h}, [w_{3}, w_{4}], w_{5}], \quad x_{i} \mapsto x_{i} \text{ for } i = 2, 3,
\end{equation*}
where $h \in C_{3}$. By Lemma \ref{leem2} (8) and Lemma \ref{leem2} (9), 
\begin{equation*}
\begin{split}
[[w_{1}, w_{2}]^{h}, [w_{3}, w_{4}], w_{5}] &= [[[w_{3}, w_{4}], [w_{1}, w_{2}]^{h}]^{-1}, w_{5}] = [[w_{3}, w_{4}], [w_{1}, w_{2}]^{h}, w_{5}]^{-1} \\
&= [[w_{3}, w_{4}]^{h^{-1}}, [w_{1}, w_{2}], w_{5}]^{-1}, 
\end{split}
\end{equation*}
and the result follows. 
\end{proof}

\begin{lemma}\label{leem4}
$H_{3}= \langle \alpha_{1, (m_{1}, m_{2}, m_{3})}, \alpha_{2, (m_{1}, m_{2}, m_{3})}, \alpha_{3, (m_{1}, m_{2}, m_{3})}: m_{1}, m_{2}, m_{3} \in \mathbb{Z}\rangle^{T_{3}}$.
\end{lemma}

\begin{proof}
Let $B_{1} = \langle \alpha_{1, (m_{1}, m_{2}, m_{3})}, \alpha_{2, (m_{1}, m_{2}, m_{3})}, \alpha_{3, (m_{1}, m_{2}, m_{3})}: m_{1}, m_{2}, m_{3} \in \mathbb{Z}\rangle^{T_{3}}$. We claim that $\beta_{i, (m_{1}, m_{2}, m_{3})} \in B_{1}$ for all $i = 1, 2, 3$ and $m_{1}, m_{2}, m_{3} \in \mathbb{Z}$ and $f_{j, i, (m_{1}, m_{2}, m_{3})} \in B_{1}$ for all $i = 1, 2, 3$, $j = 1, 2, 3, 4$ and $m_{1}, m_{2}, m_{3} \in \mathbb{Z}$.

We first show that $\beta_{i, (m_{1}, m_{2}, m_{3})} \in B_{1}$ for all $i = 1, 2, 3$ and $m_{1}, m_{2}, m_{3} \in \mathbb{Z}$. First assume that $i = 1$. For $n_{1}, n_{2}, n_{3} \in \mathbb{Z}$, by a direct calculation we get
\begin{equation*}
\begin{split}
\pi_{1, 2}\alpha_{1, (n_{1}, n_{2}, n_{3})}\pi^{-1}_{1, 2}(x_{1}) = x_{1}x_{2}[[x_{1}x_{2}, x_{2}]^{(x_{1}x_{2})^{n_{1}}x_{2}^{n_{2}}x_{3}^{n_{3}}}, [x_{1}x_{2}, x_{3}], x_{1}x_{2}]x^{-1}_{2}. 
\end{split}
\end{equation*}
and hence
\begin{equation*}
\begin{split}
\pi_{1, 2}\alpha_{1, (n_{1}, n_{2}, n_{3})}\pi^{-1}_{1, 2}(x_{1}) = x_{1}[[x_{1}x_{2}, x_{2}]^{(x_{1}x_{2})^{n_{1}}x_{2}^{n_{2}}x_{3}^{n_{3}}}, [x_{1}x_{2}, x_{3}], x_{1}x_{2}]. 
\end{split}
\end{equation*}
Therefore, using the group commutator identities $[ab, c] = [a, c]^{b}[b, c]$ and $[a, bc] = [a, c][a, b]^{c}$ and by our notation,
\begin{equation*}
\pi_{1, 2}\alpha_{1, (n_{1}, n_{2}, n_{3})}\pi^{-1}_{1, 2}(x_{1}) = x_{1}[A^{x_{2}(x_{1}x_{2})^{n_{1}}x_{2}^{n_{2}}x_{3}^{n_{3}}}, B^{x_{2}}C, x_{1}x_{2}].
\end{equation*}
Hence, by Lemma \ref{leem2} (3) and Lemma \ref{leem2} (6),
\begin{equation*}
\pi_{1, 2}\alpha_{1, (n_{1}, n_{2}, n_{3})}\pi^{-1}_{1, 2}(x_{1}) = 
x_{1}[A^{x_{2}(x_{1}x_{2})^{n_{1}}x_{2}^{n_{2}}x_{3}^{n_{3}}}, B^{x_{2}}, x_{1}x_{2}][A^{x_{2}(x_{1}x_{2})^{n_{1}}x_{2}^{n_{2}}x_{3}^{n_{3}}}, C, x_{1}x_{2}].
\end{equation*}
Since, by Lemma \ref{leem2} (5) and Lemma \ref{leem2} (9),
\begin{equation*}
[A^{x_{2}(x_{1}x_{2})^{n_{1}}x_{2}^{n_{2}}x_{3}^{n_{3}}}, B^{x_{2}}, x_{1}x_{2}] = [A^{x_{1}^{n_{1}}x_{2}^{n_{1}+n_{2}}x_{3}^{n_{3}}}, B, x_{1}x_{2}],
\end{equation*}
and by Lemma \ref{leem2} (5), 
\begin{equation*}
[A^{x_{2}(x_{1}x_{2})^{n_{1}}x_{2}^{n_{2}}x_{3}^{n_{3}}}, C, x_{1}x_{2}] = [A^{x_{1}^{n_{1}}x_{2}^{n_{1}+n_{2}+1}x_{3}^{n_{3}}}, C, x_{1}x_{2}],
\end{equation*}
it follows from the above equation that 
\begin{equation*}
\pi_{1, 2}\alpha_{1, (n_{1}, n_{2}, n_{3})}\pi_{1, 2}^{-1}(x_{1})  =
x_{1}[A^{x_{1}^{n_{1}}x_{2}^{n_{1}+n_{2}}x_{3}^{n_{3}}}, B, x_{1}x_{2}][A^{x_{1}^{n_{1}}x_{2}^{n_{1}+n_{2}+1}x_{3}^{n_{3}}}, C, x_{1}x_{2}].
\end{equation*}
Write $W_{1} = [A^{x_{1}^{n_{1}}x_{2}^{n_{1}+n_{2}}x_{3}^{n_{3}}}, B]$ and $W_{2} = [A^{x_{1}^{n_{1}}x_{2}^{n_{1}+n_{2}+1}x_{3}^{n_{3}}}, C]$. Hence, by Lemma \ref{leem2} (7), $[W_{1}, x_{1}x_{2}] = [W_{1}, x_{1}][W_{1}, x_{2}]$ and similarly $[W_{2}, x_{1}x_{2}] = [W_{2}, x_{1}][W_{2}, x_{2}]$.
Thus,
\begin{equation*}
\begin{split}
\pi_{1, 2}\alpha_{1, (n_{1}, n_{2}, n_{3})}\pi^{-1}_{1, 2}(x_{1}) &=x_{1}[A^{x_{1}^{n_{1}}x_{2}^{n_{1}+n_{2}}x_{3}^{n_{3}}}, B, x_{1}][A^{x_{1}^{n_{1}}x_{2}^{n_{1}+n_{2}}x_{3}^{n_{3}}}, B, x_{2}]\\&[A^{x_{1}^{n_{1}}x_{2}^{n_{1}+n_{2}+1}x_{3}^{n_{3}}}, C, x_{1}][A^{x_{1}^{n_{1}}x_{2}^{n_{1}+n_{2}+1}x_{3}^{n_{3}}}, C, x_{2}]
\end{split}
\end{equation*}
and hence
\begin{equation*}
\pi_{1, 2}\alpha_{1, (n_{1}, n_{2}, n_{3})}\pi^{-1}_{1, 2}(x_{1}) = \alpha_{1, (n_{1}, n_{1}+n_{2}, n_{3})}\alpha_{2, (n_{1}, n_{1}+n_{2}, n_{3})}\beta_{1, (n_{1}, n_{1}+n_{2}+1, n_{3})}\beta_{2, (n_{1}, n_{1}+n_{2}+1, n_{3})}(x_{1}).
\end{equation*}
Since also 
\begin{equation*}
\pi_{1, 2}\alpha_{j, (n_{1}, n_{2}, n_{3})}\pi^{-1}_{1, 2}(x_{\ell}) = \beta_{j, (n_{1}, n_{1}+n_{2}+1, n_{3})}(x_{\ell}) =x_{\ell} \text{ for all } j = 1, 2 \text{ and } \ell = 2, 3, 
\end{equation*}
we have
\begin{equation*}
\pi_{1, 2}\alpha_{1, (n_{1}, n_{2}, n_{3})}\pi^{-1}_{1, 2} = \alpha_{1, (n_{1}, n_{1}+n_{2}, n_{3})}\alpha_{2, (n_{1}, n_{1}+n_{2}, n_{3})}\beta_{1, (n_{1}, n_{1}+n_{2}+1, n_{3})}\beta_{2, (n_{1}, n_{1}+n_{2}+1, n_{3})}.
\end{equation*}
Therefore, for $n_{1}  = m_{1}$, $n_{2} = -m_{1} + m_{2} - 1$  and $n_{3} = m_{3}$, it follows from the above equation that
\begin{equation*}
\beta_{1, (m_{1}, m_{2}, m_{3})}\beta_{2, (m_{1}, m_{2}, m_{3})} = \alpha^{-1}_{1, (m_{1}, m_{2} - 1, m_{3})}\alpha^{-1}_{2, (m_{1}, m_{2} - 1, m_{3})}\pi_{1, 2}\alpha_{1, (m_{1}, -m_{1} + m_{2} - 1, m_{3})}\pi^{-1}_{1, 2}.
\end{equation*}
Assume now that $i \neq 1$. An analogous argument gives
\begin{equation*}
\begin{split}
&\beta_{2, (m_{1}, m_{2}, m_{3})} = \alpha^{-1}_{2, (m_{1}, m_{2} - 1, m_{3})}\pi_{1, 2}\alpha_{2, (m_{1}, -m_{1} + m_{2} - 1, m_{3})}\pi^{-1}_{1, 2},\\
&\beta_{3, (m_{1}, m_{2}, m_{3})} = \alpha^{-1}_{3, (m_{1}, m_{2} - 1, m_{3})}\pi_{1, 2}\alpha_{3, (m_{1}, -m_{1} + m_{2} - 1, m_{3})}\pi^{-1}_{1, 2}.
\end{split}
\end{equation*}
Thus, by the above equations, $\beta_{i, (m_{1}, m_{2}, m_{3})} \in B_{1}$ for all $i = 1, 2, 3$. 

We show now that $f_{j, i, (m_{1}, m_{2}, m_{3})} \in B_{1}$ for all $i = 1, 2, 3$, $j = 1, \ldots, 4$ and $m_{1}, m_{2}, m_{3} \in \mathbb{Z}$. By calculations similar to the above, we get
\begin{equation*}
\begin{split}
&f_{1, 1, (m_{1}, m_{2}, m_{3})} = \alpha^{-1}_{1, (m_{1}, m_{2} - 1, m_{3})}\pi_{3, 2}\alpha_{1, (m_{1}, m_{2} - m_{3}, m_{3})}\pi^{-1}_{3, 2},\\
&f_{1, 2, (m_{1}, m_{2}, m_{3})} = \alpha^{-1}_{2, (m_{1}, m_{2} - 1, m_{3})}\pi_{3, 2}\alpha_{2, (m_{1}, m_{2} - m_{3}, m_{3})}\pi^{-1}_{3, 2},\\
&f_{1, 2, (m_{1}, m_{2}, m_{3})}f_{1, 3, (m_{1}, m_{2}, m_{3})} = \alpha^{-1}_{2, (m_{1}, m_{2} - 1, m_{3})}\alpha^{-1}_{3, (m_{1}, m_{2} - 1, m_{3})}\pi_{3, 2}\alpha_{3, (m_{1}, m_{2} - m_{3}, m_{3})}\pi^{-1}_{3, 2}, \\
&f_{4, 2, (m_{1}, m_{2}, m_{3})} = \beta_{2, (m_{1}, m_{2}, m_{3}+1)}\pi_{1, 3}\beta^{-1}_{2, (m_{1}, m_{2}, m_{3} - m_{1})}\pi^{-1}_{1, 3},\\
&f_{4, 3, (m_{1}, m_{2}, m_{3})} = \beta_{3, (m_{1}, m_{2}, m_{3}+1)}\pi_{1, 3}\beta^{-1}_{3, (m_{1}, m_{2}, m_{3} - m_{1})}\pi^{-1}_{1, 3},\\
&f_{4, 1, (m_{1}, m_{2}, m_{3})}f_{4, 3, (m_{1}, m_{2}, m_{3})} = \beta_{3, (m_{1}, m_{2}, m_{3}+1)}\beta_{1, (m_{1}, m_{2}, m_{3}+1)}\pi_{1, 3}\beta^{-1}_{1, (m_{1}, m_{2}, m_{3} - m_{1})}\pi^{-1}_{1, 3}.
\end{split}
\end{equation*}
Furthermore, it is straightforward to verify that
\begin{equation*}
\begin{split}
&f_{2, 1, (m_{1}, m_{2}, m_{3})} = \sigma^{-1}_{2, 3}f_{1, 1, (m_{1}, m_{3}, m_{2})}\sigma_{2, 3},\\
&f_{2, 2, (m_{1}, m_{2}, m_{3})} = \sigma^{-1}_{2, 3}f_{1, 3, (m_{1}, m_{3}, m_{2})}\sigma_{2, 3},\\
&f_{2, 3, (m_{1}, m_{2}, m_{3})} = \sigma^{-1}_{2, 3}f_{1, 2, (m_{1}, m_{3}, m_{2})}\sigma_{2, 3},\\ 
&f_{3, 1, (m_{1}, m_{2}, m_{3})} = \sigma^{-1}_{2, 3}\beta^{-1}_{1, (m_{1}, m_{3}, m_{2})}\sigma_{2, 3},\\
&f_{3, 2, (m_{1}, m_{2}, m_{3})} = \sigma^{-1}_{2, 3}\beta^{-1}_{3, (m_{1}, m_{3}, m_{2})}\sigma_{2, 3},\\ 
&f_{3, 3, (m_{1}, m_{2}, m_{3})} = \sigma^{-1}_{2, 3}\beta^{-1}_{2, (m_{1}, m_{3}, m_{2})}\sigma_{2, 3}.
\end{split}
\end{equation*}
Hence, by the above equations, it follows that $f_{j, i, (m_{1}, m_{2}, m_{3})} \in B_{1}$ for all $i = 1, 2, 3$ and $j = 1, \ldots, 4$. 

Thus, our claim is true, that is, $\beta_{i, (m_{1}, m_{2}, m_{3})} \in B_{1}$ for all $i = 1, 2, 3$ and $f_{j, i, (m_{1}, m_{2}, m_{3})} \in B_{1}$ for all $i = 1, 2, 3$ and $j = 1, 2, 3, 4$. Hence, it follows from Lemma \ref{leem3} that $H_{3} = B_{1}$.
\end{proof}

\begin{lemma}\label{leem5}
$H_{3}= \langle \alpha_{1, (m_{1}, m_{2}, m_{3})}, \alpha_{2, (m_{1}, m_{2}, m_{3})}, \alpha_{3, (m_{1}, m_{2}, m_{3})}: m_{1}, m_{2}, m_{3} \geq 0\rangle^{T_{3}}$.
\end{lemma}

\begin{proof}
Let $B_{2} = \langle \alpha_{1, (m_{1}, m_{2}, m_{3})}, \alpha_{2, (m_{1}, m_{2}, m_{3})}, \alpha_{3, (m_{1}, m_{2}, m_{3})}: m_{1}, m_{2}, m_{3} \geq 0\rangle^{T_{3}}$. By Lemma \ref{leem4}, it suffices to prove that $\alpha_{i, (m_{1}, m_{2}, m_{3})} \in B_{2}$ for all $i = 1, 2, 3$ and $m_{1}, m_{2}, m_{3} \in \mathbb{Z}$. Let $m_{1}, m_{2}, m_{3} \in \mathbb{Z}$. Throughout the proof, we write $h = x_{2}^{m_{2}}x_{3}^{m_{3}}$. By a direct calculation, we get
\begin{equation*}
k^{-1}_{1}\alpha_{1, (m_{1}, m_{2}, m_{3})}k_{1}(x_{1}) = [[x^{-1}_{1}, x_{2}]^{x^{-m_{1}}_{1}h}, [x^{-1}_{1}, x_{3}], x^{-1}_{1}]^{-1}x_{1},
\end{equation*}
and thus, 
\begin{equation*}
\begin{split}
k^{-1}_{1}\alpha_{1, (m_{1}, m_{2}, m_{3})}k_{1}(x_{1}) = x_{1}[[x^{-1}_{1}, x_{2}]^{x^{-m_{1}}_{1}h}, [x^{-1}_{1}, x_{3}], x^{-1}_{1}]^{-1}.
\end{split}
\end{equation*}
By our notation, by the group commutator identity $[a^{-1}, b] = ([a, b]^{-1})^{a^{-1}}$ and by Lemma \ref{leem2} (8), we get
\begin{equation*}
k_{1}^{-1}\alpha_{1, (m_{1}, m_{2}, m_{3})}k_{1}(x_{1}) = x_{1}[(A^{-1})^{x^{-m_{1}-1}_{1}h}, (B^{-1})^{x^{-1}_{1}}, x^{-1}_{1}]^{-1} = x_{1}[A^{x^{-m_{1}-1}_{1}h}, B^{x^{-1}_{1}}, x_{1}].
\end{equation*}
Since, by Lemma \ref{leem2} (9) and Lemma \ref{leem2} (5), 
\begin{equation*}
[A^{x^{-m_{1}-1}_{1}h}, B^{x^{-1}_{1}}, x_{1}] = [A^{x^{-m_{1}-1}_{1}hx_{1}}, B, x_{1}] = [A^{x^{-m_{1}}_{1}h}, B, x_{1}] = [A^{x^{-m_{1}}_{1}h}, B, x_{1}],
\end{equation*}
it follows from the above equation that 
\begin{equation*}
k^{-1}_{1}\alpha_{1, (m_{1}, m_{2}, m_{3})}k_{1}(x_{1}) = x_{1}[A^{x^{-m_{1}}_{1}h}, B, x_{1}] = \alpha_{1, (-m_{1}, m_{2}, m_{3})}(x_{1}). 
\end{equation*}
Furthermore, since 
\begin{equation*}
k^{-1}_{1}\alpha_{1, (m_{1}, m_{2}, m_{3})}k_{1}(x_{j}) = \alpha_{1, (-m_{1}, m_{2}, m_{3})}(x_{j}) = x_{j} \text{ for } j = 2, 3, 
\end{equation*}
it follows that 
\begin{equation*}
k^{-1}_{1}\alpha_{1, (m_{1}, m_{2}, m_{3})}k_{1} = \alpha_{1, (-m_{1}, m_{2}, m_{3})}.
\end{equation*}
By analogous calculations, we get
\begin{equation*}
\begin{split}
k^{-1}_{2}\alpha^{-1}_{1, (m_{1}, m_{2}-1, m_{3})}k_{2} &= \alpha_{1, (m_{1}, -m_{2}, m_{3})},\\
k^{-1}_{3}\alpha^{-1}_{1, (m_{1}, m_{2}, m_{3}+1)}k_{3} &= \alpha_{1, (m_{1}, m_{2}, -m_{3})},\\ 
k^{-1}_{1}\alpha^{-1}_{2, (m_{1}, m_{2}, m_{3})}k_{1} &= \alpha_{2, (-m_{1}, m_{2}, m_{3})},\\
k^{-1}_{2}\alpha_{2, (m_{1}, m_{2}-1, m_{3})}k_{2} &= \alpha_{2, (m_{1}, -m_{2}, m_{3})},\\
k^{-1}_{3}\alpha^{-1}_{2, (m_{1}, m_{2}, m_{3}+1)}k_{3} &= \alpha_{2, (m_{1}, m_{2}, -m_{3})},\\
k^{-1}_{1}\alpha^{-1}_{3, (m_{1}, m_{2}, m_{3})}k_{1} &= \alpha_{3, (-m_{1}, m_{2}, m_{3})},\\
k^{-1}_{2}\alpha^{-1}_{3, (m_{1}, m_{2}-1, m_{3})}k_{2} &= \alpha_{3, (m_{1}, -m_{2}, m_{3})},\\
k^{-1}_{3}\alpha_{3, (m_{1}, m_{2}, m_{3}+1)}k_{3} &= \alpha_{3, (m_{1}, m_{2}, -m_{3})}.
\end{split}
\end{equation*}
By the above equations, we deduce that $\alpha_{i, (m_{1}, m_{2}, m_{3})} \in B_{2}$ for all $i = 1, 2, 3$ and $m_{1}, m_{2}, m_{3} \in \mathbb{Z}$. Hence, by Lemma \ref{leem4}, $H_{3} \subseteq B_{2}$ and thus, $H_{3} = B_{2}$.
\end{proof}

\section{Technical Commutator Identities}\label{sec4}

In this section, we give all the computational results we will use later. 

\begin{lemma}\label{leem6}
Let $m_{1}, m_{2}, m_{3} \geq 0$ and let $\tau_{1}$ be the tame automorphism of $G_{3}$ defined by $\tau_{1}(x_{1}) = x_{2}x_{1}$, $\tau_{1}(x_{i}) = x_{i}$ for $i = 2, 3$. Then, the following hold:
\begin{enumerate}
\item 
\begin{enumerate}
\item $\alpha^{-1}_{2, (0, 0, 0)}\tau_{1}\alpha_{2, (0, 0, 0)}\tau^{-1}_{1} = k^{-1}_{1}\beta_{2, (0, 0, 0)}k_{1}$.

\item $\alpha^{-1}_{3, (0, 0, 0)}\tau_{1}\alpha_{3, (0, 0, 0)}\tau^{-1}_{1} = k^{-1}_{1}\beta_{3, (0, 0, 0)}k_{1}$.

\item $\alpha_{2, (0, 0, 0)}\alpha_{1, (0, 0, 0)}\tau_{1}\alpha^{-1}_{1, (0, 0, 0)}\tau^{-1}_{1} = k^{-1}_{1}\beta^{-1}_{1, (0, 0, 0)}\beta_{2, (0, 0, 0)}k_{1}$.
\end{enumerate}

\item
\begin{enumerate}
\item $\sigma_{2, 3}\alpha_{1, (m_{1}, m_{3}, -m_{2})}\sigma^{-1}_{2, 3} = k^{-1}_{3}k^{-1}_{1}\alpha_{1, (m_{1}, m_{2}, m_{3}+1)}k_{1}k_{3}$.

\item $\sigma_{2, 3}\alpha_{2, (m_{1}, m_{3}, -m_{2})}\sigma^{-1}_{2, 3} = k^{-1}_{3}k^{-1}_{1}\alpha_{3, (m_{1}, m_{2}, m_{3}+1)}k_{1}k_{3}$.

\item $\sigma_{2, 3}\alpha_{3, (m_{1}, m_{3}, -m_{2})}\sigma^{-1}_{2, 3} = k^{-1}_{3}k^{-1}_{1}\alpha^{-1}_{2, (m_{1}, m_{2}, m_{3}+1)}k_{1}k_{3}$.
\end{enumerate}

\item $\sigma^{-1}_{2, 3}\alpha^{-1}_{2, (0, 0, 0)}\sigma_{2, 3} = \alpha_{3, (0, 0, 0)}$.
\end{enumerate}
\end{lemma}

\begin{proof}
\begin{enumerate}
\item  By arguments similar to those in the proof of Lemma \ref{leem4} and by our notation, we get
\begin{equation*}
\begin{split}
\tau_{1}\alpha_{2, (0, 0, 0)}\tau^{-1}_{1}(x_{1}) &= x_{1}[A, B, x_{2}][A^{x^{-1}_{1}}, C, x_{2}], \\
\tau_{1}\alpha_{3, (0, 0, 0)}\tau^{-1}_{1}(x_{1}) &= x_{1}[A, B, x_{3}][A^{x^{-1}_{1}}, C, x_{3}], \\
\tau_{1}\alpha_{1, (0, 0, 0)}\tau^{-1}_{1}(x_{1}) &= x_{1}[A, B, x_{1}][A, B, x_{2}][A^{x^{-1}_{1}}, C, x_{1}][A^{x^{-1}_{1}}, C, x_{2}],
\end{split}
\end{equation*}
and, by arguments similar to those in the proof of Lemma \ref{leem5}, 
\begin{equation*}
\begin{split}
k^{-1}_{1}\beta_{2, (0, 0, 0)}k_{1}(x_{1}) &= x_{1}[A^{x^{-1}_{1}}, C, x_{2}], \\
k^{-1}_{1}\beta_{3, (0, 0, 0)}k_{1}(x_{1}) &= x_{1}[A^{x^{-1}_{1}}, C, x_{3}], \\
k^{-1}_{1}\beta_{1, (0, 0, 0)}k_{1}(x_{1}) &= x_{1}[A^{x^{-1}_{1}}, C, x_{1}]^{-1}.
\end{split}
\end{equation*}
Thus, the result may be easily deduced.

\item By the proof of Lemma \ref{leem5}, 
\begin{equation*}
k^{-1}_{3}k^{-1}_{1}\alpha_{1, (m_{1}, m_{2}, m_{3}+1)}k_{1}k_{3} = k^{-1}_{3}\alpha_{1, (-m_{1}, m_{2}, m_{3}+1)}k_{3} = \alpha^{-1}_{1, (-m_{1}, m_{2}, -m_{3})}
\end{equation*}
Furthermore, 
\begin{equation*}
\sigma_{2, 3}\alpha_{1, (m_{1}, m_{3}, -m_{2})}\sigma^{-1}_{2, 3}(x_{1}) = \sigma_{2, 3}(x_{1}[A^{x_{1}^{m_{1}}x_{2}^{m_{3}}x_{3}^{-m_{2}}}, B, x_{1}]) = x_{1}[B^{x_{1}^{m_{1}}x_{3}^{m_{3}}x_{2}^{-m_{2}}}, A, x_{1}].
\end{equation*}
Since, by Lemma \ref{leem2} (8), Lemma \ref{leem2} (9) and Lemma \ref{leem2} (5),
\begin{equation*}
\begin{split}
[B^{x_{1}^{m_{1}}x_{3}^{m_{3}}x_{2}^{-m_{2}}}, A, x_{1}] &= [A, B^{x_{1}^{m_{1}}x_{3}^{m_{3}}x_{2}^{-m_{2}}}, x_{1}]^{-1} = [A^{(x_{1}^{m_{1}}x_{3}^{m_{3}}x_{2}^{-m_{2}})^{-1}}, B, x_{1}]^{-1}\\
& = [A^{x_{1}^{-m_{1}}x_{2}^{m_{2}}x_{3}^{-m_{3}}}, B, x_{1}]^{-1},
\end{split}
\end{equation*}
it follows that 
\begin{equation*}
\sigma_{2, 3}\alpha_{1, (m_{1}, m_{3}, -m_{2})}\sigma^{-1}_{2, 3}(x_{1}) = x_{1}[A^{x_{1}^{-m_{1}}x_{2}^{m_{2}}x_{3}^{-m_{3}}}, B, x_{1}]^{-1} = \alpha^{-1}_{1, (-m_{1}, m_{2}, -m_{3})}(x_{1}). 
\end{equation*}
Since also 
\begin{equation*}
\sigma_{2, 3}\alpha_{1, (m_{1}, m_{3}, -m_{2})}\sigma^{-1}_{2, 3}(x_{j}) = \alpha^{-1}_{1, (-m_{1}, m_{2}, -m_{3})}(x_{j}) = x_{j} \text{ for } j = 2, 3, 
\end{equation*}
we get 
\begin{equation*}
\sigma_{2, 3}\alpha_{1, (m_{1}, m_{3}, -m_{2})}\sigma^{-1}_{2, 3} = \alpha^{-1}_{1, (-m_{1}, m_{2}, -m_{3})}. 
\end{equation*}
Thus, 
\begin{equation*}
\sigma_{2, 3}\alpha_{1, (m_{1}, m_{3}, -m_{2})}\sigma^{-1}_{2, 3} = k^{-1}_{3}k^{-1}_{1}\alpha_{1, (m_{1}, m_{2}, m_{3}+1)}k_{1}k_{3}. 
\end{equation*}
Parts 2 (b) and 2 (c) follow similarly.

\item This follows from a direct calculation.
\end{enumerate}
\end{proof}

Recall that, by Lemma \ref{leem1} (1), any endomorphism of $G_{n}$ that induces an automorphism of $M_{n}$ is an automorphism. Let $\theta$ be the automorphism of $G_{3}$ satisfying the conditions
\begin{equation*}
\begin{split}
\theta(x_{1}) &\equiv x_{1}w \pmod {[G_{3}^{\prime\prime}, G_{3}]}, \text{ where } w \in G_{3}^{\prime\prime}, \\
\theta(x_{\ell}) &\equiv x_{\ell}\pmod {[G_{3}^{\prime\prime}, G_{3}]} \text{ for } \ell = 2, 3.
\end{split}
\end{equation*}
For $k \in \{1, 2, 3\}$, we associate to $\theta$ an automorphism $(\theta)_{k}$ defined by 
\begin{equation*}
(\theta)_{k}(x_{1}) = x_{1}[w, x_{k}], (\theta)_{k}(x_{\ell}) = x_{\ell} \text{ for } \ell = 2, 3. 
\end{equation*}
Moreover, for $i, j \in \{1, 2, 3\}$, with $i \neq j$, let $\tau_{i,j}$ be the tame automorphism of $G_{3}$ defined by 
\begin{equation*}
\tau_{i, j}(x_{i}) = x_{i}[x_{i}, x_{j}], \tau_{i, j}(x_{\ell}) = x_{\ell} \text{ for } \ell \neq i. 
\end{equation*}
Using this notation, we prove the following result.

\begin{lemma}\label{leem7}
\begin{enumerate}
\item 
\begin{enumerate}

\item For all $i, j \in \{1, 2, 3\}$, with $i \neq j$, the automorphism $\tau_{i, j}\theta \tau^{-1}_{i, j}$ satisfies the conditions
\begin{equation*}
\begin{split}
\tau_{i, j}\theta \tau^{-1}_{i, j}(x_{1}) &\equiv x_{1}\tau_{i, j}(w) \pmod {[G_{3}^{\prime\prime}, G_{3}]}\\
\tau_{i, j}\theta \tau^{-1}_{i, j}(x_{\ell}) &\equiv x_{\ell} \pmod {[G_{3}^{\prime\prime}, G_{3}]} \text{ for } \ell = 2, 3.
\end{split}
\end{equation*}
\item For any $k \in \{1, 2, 3\}$ and $i, j \in \{1, 2, 3\}$, with $i \neq j$, 
\begin{equation*}
\tau_{i, j}(\theta)_{k} \tau^{-1}_{i, j} = (\tau_{i, j}\theta \tau^{-1}_{i, j})_{k}.
\end{equation*}
\end{enumerate}

\item Let $f$ and $g$ be the automorphisms of $G_{3}$ satisfying the conditions
\begin{equation*}
\begin{split}
f(x_{1}) &= x_{1}w_{1}, \text{ where } w_{1} \in G_{3}^{\prime\prime}, f(x_{\ell}) = x_{\ell} \text{ for } \ell = 2, 3, \\
g(x_{1}) &= x_{1}w_{2}, \text{ where } w_{2} \in G_{3}^{\prime\prime}, g(x_{\ell}) = x_{\ell} \text{ for } \ell = 2, 3.
\end{split}
\end{equation*}
Then:
\begin{enumerate}
\item $(fg)_{k} = (f)_{k}(g)_{k}$ for all $k \in \{1, 2, 3\}$. 

\item Let $\phi_{3}$ be the group homomorphism from ${\rm Aut}(G_{3})$ into ${\rm Aut}(C_{3})$, induced from the natural mapping from $G_{3}$ onto $C_{3}$. If $\phi_{3}(f) = \phi_{3}(g)$, then $(f)_{k} = (g)_{k}$ for all $k \in \{1, 2, 3\}$.
\end{enumerate}
\end{enumerate}
\end{lemma}

\begin{proof}
\begin{enumerate}

\item 

\begin{enumerate}

\item Note that $\tau_{i, j}(x_{i}) = x^{-1}_{j}x_{i}x_{j}$ and $\tau^{-1}_{i, j}(x_{i}) = x_{j}x_{i}x^{-1}_{j}$. First assume that $i \neq 1$. Then, it follows directly that 
\begin{equation*}
\tau_{i, j}\theta \tau^{-1}_{i, j}(x_{1}) \equiv  x_{1}\tau_{i, j}(w) \pmod {[G_{3}^{\prime\prime}, G_{3}]}. 
\end{equation*}
Let $\ell \neq 1$. In the case that $\ell \neq i$ and in the case that $\ell = i$ and $j \neq 1$, it follows directly that $\tau_{i, j}\theta \tau^{-1}_{i, j}(x_{\ell}) \equiv  x_{\ell} \pmod {[G_{3}^{\prime\prime}, G_{3}]}$. Thus, we may assume that $\ell = i$ and $j = 1$. Then,  $\tau^{-1}_{\ell, 1}(x_{\ell}) = x_{\ell}[x_{\ell}, x^{-1}_{1}]$ and thus, $\tau_{\ell, 1}\theta \tau^{-1}_{\ell, 1}(x_{\ell}) \equiv \tau_{\ell,1}(x_{\ell}[x_{\ell}, w^{-1}x^{-1}_{1}])\pmod {[G_{3}^{\prime\prime}, G_{3}]}$. By the group commutator identity $[a, bc] = [a, c][a, b][a, b, c]$ and since $w \in G_{3}^{\prime\prime}$, it follows directly that $[x_{\ell}, w^{-1}x^{-1}_{1}] \equiv [x_{\ell}, x^{-1}_{1}] \pmod {[G_{3}^{\prime\prime}, G_{3}]}$ and therefore, $\tau_{\ell, 1}\theta \tau^{-1}_{\ell, 1}(x_{\ell}) \equiv \tau_{\ell, 1}(x_{\ell}[x_{\ell}, x^{-1}_{1}]) \equiv x_{\ell} \pmod {[G_{3}^{\prime\prime}, G_{3}]}$. Thus, 
\begin{equation*}
\tau_{\ell, 1}\theta \tau^{-1}_{\ell, 1}(x_{\ell}) \equiv x_{\ell} \pmod {[G_{3}^{\prime\prime}, G_{3}]} \text{ for } \ell = 2, 3.
\end{equation*}
Assume now that $i = 1$. Then, by a direct calculation, 
\begin{equation*}
\tau_{1, j}\theta \tau^{-1}_{1, j}(x_{1}) \equiv (x_{1}^{x_{j}}\tau_{1, j}(w))^{x^{-1}_{j}} \pmod {[G_{3}^{\prime\prime}, G_{3}]},
\end{equation*}
and thus, 
\begin{equation*}
\tau_{1, j}\theta \tau^{-1}_{1, j}(x_{1}) \equiv  x_{1}(\tau_{1 ,j}(w))^{x^{-1}_{j}} \equiv x_{1}\tau_{1, j}(w)[\tau_{1, j}(w), x^{-1}_{j}] \pmod {[G_{3}^{\prime\prime}, G_{3}]}.
\end{equation*}
Since $w \in G_{3}^{\prime\prime}$ and $G_{3}^{\prime\prime}$ is a fully invariant subgroup of $G_{3}$, we have $\tau_{i, j}(w) \in G^{\prime\prime}_{3}$. Hence, $[\tau_{i,j}(w), x^{-1}_{j}] \in [G_{3}^{\prime\prime}, G_{3}]$ and thus, by the above equation, 
\begin{equation*}
\tau_{1, j}\theta \tau^{-1}_{1, j}(x_{1}) \equiv x_{1}\tau_{1, j}(w) \pmod {[G_{3}^{\prime\prime}, G_{3}]}. 
\end{equation*}
Furthermore, it is straightforward to verify that 
\begin{equation*}
\tau_{1, j}\theta \tau^{-1}_{1, j}(x_{\ell}) \equiv x_{\ell} \pmod {[G_{3}^{\prime\prime}, G_{3}]} \text{ for } \ell =2, 3.
\end{equation*}
\item Let $k \in \{1, 2, 3\}$. First assume that $i = 1$. Then, by a direct calculation, 
\begin{equation*}
\tau_{1, j}(\theta)_{k} \tau^{-1}_{1, j}(x_{1}) = (x_{1}^{x_{j}}[\tau_{1, j}(w), \tau_{1, j}(x_{k})])^{x^{-1}_{j}}
\end{equation*}
and thus, 
\begin{equation*}
\tau_{1, j}(\theta)_{k} \tau^{-1}_{1, j}(x_{1}) = x_{1}[\tau_{1, j}(w), \tau_{1, j}(x_{k})]^{x^{-1}_{j}} = x_{1}[\tau_{1, j}(w), \tau_{1, j}(x_{k})]. 
\end{equation*}
Since $\tau_{1, j}(w) \in G_{3}^{\prime\prime}$, $\tau_{1, j}(x_{k}) \equiv x_{k} \mod G_{3}^{\prime}$ and $[G_{3}^{\prime\prime}, G_{3}^{\prime}] = \{1\}$, it follows that $[\tau_{1, j}(w), \tau_{1, j}(x_{k})] = [\tau_{1, j}(w), x_{k}]$. Thus, $\tau_{i,j}(\theta)_{k} \tau^{-1}_{i,j}(x_{1}) = x_{1}[\tau_{i,j}(w), x_{k}]$. Furthermore, $\tau_{1, j}(\theta)_{k} \tau^{-1}_{1, j}(x_{\ell}) = x_{\ell}$ for $\ell = 2, 3$.  Assume now that $i \neq 1$. Then, 
\begin{equation*}
\tau_{i,j}(\theta)_{k} \tau^{-1}_{i, j}(x_{1}) = \tau_{i, j}(\theta)_{k} (x_{1}) = \tau_{i, j}(x_{1}[w, x_{k}]) = x_{1}[\tau_{i, j}(w), \tau_{i, j}(x_{k})].
\end{equation*}
Hence, by analogous arguments, $\tau_{i, j}(\theta)_{k} \tau^{-1}_{i, j}(x_{1}) = x_{1}[\tau_{i, j}(w), x_{k}]$. Furthermore, by arguments similar to those in the proof of Lemma \ref{leem7} (1) (a), it is easily verified that $\tau_{i, j}(\theta)_{k} \tau^{-1}_{i, j}(x_{\ell}) = x_{\ell}$ for $\ell = 2, 3$. Thus, for all $i, j \in \{1, 2, 3\}$, with $i \neq j$,
\begin{equation*}
\tau_{i, j}(\theta)_{k} \tau^{-1}_{i, j}(x_{1}) = x_{1}[\tau_{i, j}(w), x_{k}], \tau_{i, j}(\theta)_{k} \tau^{-1}_{i, j}(x_{\ell}) = x_{\ell} \text{ for } \ell = 2, 3.
\end{equation*}
Since, by Lemma \ref{leem7} (1) (a) and the above definition, 
\begin{equation*}
(\tau_{i, j}\theta \tau^{-1}_{i, j})_{k}(x_{1}) = x_{1}[\tau_{i, j}(w), x_{k}], 
(\tau_{i, j}\theta \tau^{-1}_{i, j})_{k}(x_{\ell}) = x_{\ell} \text{ for } \ell = 2, 3,
\end{equation*}
it follows that $\tau_{i, j}(\theta)_{k} \tau^{-1}_{i, j} = (\tau_{i, j}\theta \tau^{-1}_{i, j})_{k}$.
\end{enumerate}

\item 

\begin{enumerate}

\item By Lemma \ref{extrlem1} (1), $fg(x_{1}) = x_{1}w_{1}w_{2}$ and $fg(x_{\ell}) = x_{\ell}$ for $\ell = 2, 3$. Let $k \in \{1, 2, 3\}$. Then, by Lemma \ref{leem2} (6),
\begin{equation*}
\begin{split}
(fg)_{k}(x_{1}) &= x_{1}[w_{1}w_{2}, x_{k}] = x_{1}[w_{1}, x_{k}][w_{2}, x_{k}] = (f)_{k}(g)_{k}(x_{1})\\
(fg)_{k}(x_{\ell}) &= x_{\ell} = (f)_{k}(g)_{k}(x_{\ell}) \text{ for }\ell = 2, 3.
\end{split}
\end{equation*}
\item Since $\phi_{3}(f) = \phi_{3}(g)$, $g^{-1}f \in {\rm Ker}(\phi_{3}) = H_{3}$ and hence, $f = gh$ for some $h \in H_{3}$. Therefore, $f(x_{1}) \equiv g(x_{1}) \pmod {[G_{3}^{\prime\prime}, G_{3}]}$, that is, $w_{1} \equiv w_{2} \pmod {[G_{3}^{\prime\prime}, G_{3}]}$, and $f(x_{\ell}) = x_{\ell} = g(x_{\ell})$ for $\ell = 2, 3$. Hence, for any $k \in \{1, 2, 3\}$, $(f)_{k}(x_{1}) = x_{1}[w_{1}, x_{k}] = x_{1}[w_{2}, x_{k}] = (g)_{k}(x_{1})$ and $(f)_{k}(x_{\ell}) = x_{\ell} = (g)_{k}(x_{\ell})$ for $\ell = 2, 3$. Thus, $(f)_{k} = (g)_{k}$ for all $k \in \{1, 2, 3\}$.
\end{enumerate}
\end{enumerate}
\end{proof}

Lemma \ref{leem7} allows us to transfer results proved in \cite{kof1} for automorphisms of $C_{3}$ to corresponding results for automorphisms of $G_{3}$.

\begin{lemma}\label{leem8}
Let $m_{1}, m_{2}, m_{3} \geq 0$. Let $\tau_{1}$ be the tame automorphism of $G_{3}$ defined by $\tau_{1}(x_{1}) = x_{2}x_{1}$, $\tau_{1}(x_{i}) = x_{i}$ for $i = 2, 3$. For $i, j \in \{1, 2, 3\}$, let $\tau_{i,j}$ be the tame automorphism of $G_{3}$ defined by $\tau_{i, j}(x_{i}) = x_{i}[x_{i}, x_{j}]$, $\tau_{i, j}(x_{\ell}) = x_{\ell}$ for $\ell \neq i$. Then, for all $i \in \{1, 2, 3\}$, the following hold:
\begin{enumerate}

\item $\alpha^{-1}_{i, (m_{1}+1, 0, 0)}\tau_{3,2}\alpha_{i, (m_{1}+1, 0, 0)}\tau^{-1}_{3,2} = \beta^{-1}_{i, (m_{1}+1, 0, 0)}\beta_{i, (m_{1}, 0, 0)}$.

\item  $\tau_{1, 2}\beta_{i, (m_{1}, m_{2}, 0)}\tau^{-1}_{1, 2} = \beta_{i, (m_{1}, m_{2}+1, 0)}$. 

\item $\tau_{1, 2}\beta_{i, (m_{1}, m_{2}, m_{3})}\tau^{-1}_{1, 2} = \beta_{i, (m_{1}, m_{2}+1, m_{3})}$.

\item $\beta^{-1}_{i, (m_{1}, m_{2}+1, 0)}\tau_{2, 1}\beta_{i, (m_{1}, m_{2}+1, 0)}\tau^{-1}_{2, 1} = \tau_{2,1}\alpha_{i, (m_{1}, m_{2}+1, 0)}\alpha^{-1}_{i, (m_{1}, m_{2}, 0)}\tau^{-1}_{2,1}$.

\item $\beta^{-1}_{i, (m_{1}, m_{2}+1, m_{3})}\tau_{2, 1}\beta_{i, (m_{1}, m_{2}+1, m_{3})}\tau^{-1}_{2, 1} = \tau_{2,1}\alpha_{i, (m_{1}, m_{2}+1, m_{3})}\alpha^{-1}_{i, (m_{1}, m_{2}, m_{3})}\tau^{-1}_{2,1}$.

\item $\tau_{3, 1}\alpha_{i, (m_{1}+1, 0, 0)}\tau^{-1}_{3, 1} = \alpha_{i, (m_{1}, 0, 0)}$.

\item 
\begin{enumerate}
\item $\beta^{-1}_{1, (m_{1}, m_{3}, 0)}\tau_{1, 3}\beta_{1, (m_{1}, m_{3}, 0)} \tau^{-1}_{1, 3} = \sigma_{2, 3}\beta^{-1}_{1, (m_{1}, 0, m_{3}+1)}\beta_{1, (m_{1}, 0, m_{3})}\sigma^{-1}_{2, 3}$.

\item $\beta^{-1}_{3, (m_{1}, m_{3}, 0)}\tau_{1, 3}\beta_{3, (m_{1}, m_{3}, 0)} \tau^{-1}_{1, 3} = \sigma_{2, 3}\beta^{-1}_{2, (m_{1}, 0, m_{3}+1)}\beta_{2, (m_{1}, 0, m_{3})}\sigma^{-1}_{2, 3}$.

\item $\beta^{-1}_{2, (m_{1}, m_{3}, 0)}\tau_{1, 3}\beta_{2, (m_{1}, m_{3}, 0)} \tau^{-1}_{1, 3} = \sigma_{2, 3}\beta^{-1}_{3, (m_{1}, 0, m_{3}+1)}\beta_{3, (m_{1}, 0, m_{3})}\sigma^{-1}_{2, 3}$.
\end{enumerate}
\end{enumerate}
\end{lemma}

\begin{proof}
For $m_{1}, m_{2}, m_{3} \in \mathbb{Z}$ and $g =x_{1}^{m_{1}}x_{2}^{m_{2}}x_{3}^{m_{3}}$, we write $\alpha_{(m_{1}, m_{2}, m_{3})}$ and $\beta_{(m_{1}, m_{2}, m_{3})}$ for the endomorphisms of $G_{3}$ defined by
\begin{equation*}
\begin{split}
\alpha_{(m_{1}, m_{2}, m_{3})}(x_{1}) &= x_{1}[[x_{1}, x_{2}]^{g}, [x_{1}, x_{3}]], \alpha_{(m_{1}, m_{2}, m_{3})}(x_{i}) = x_{i}, i = 2, 3,\\
\beta_{(m_{1}, m_{2}, m_{3})}(x_{1}) &= x_{1}[[x_{1}, x_{2}]^{g}, [x_{2}, x_{3}]], \beta_{(m_{1}, m_{2}, m_{3})}(x_{i}) = x_{i}, i = 2, 3.
\end{split}
\end{equation*}
Note that, by Lemma \ref{leem1} (1), $\alpha_{(m_{1}, m_{2}, m_{3})}$ and $\beta_{(m_{1}, m_{2}, m_{3})}$ are automorphisms of $G_{3}$. It follows from \cite[proof of Proposition 6]{kof1} that
\begin{enumerate}

\item $\phi_{3}(\alpha^{-1}_{(m_{1}+1, 0, 0)}\tau_{3,2}\alpha_{(m_{1}+1, 0, 0)}\tau^{-1}_{3,2}) = \phi_{3}(\beta^{-1}_{(m_{1}+1, 0, 0)}\beta_{(m_{1}, 0, 0)})$,

\item  $\phi_{3}(\tau_{1, 2}\beta_{(m_{1}, m_{2}, 0)}\tau^{-1}_{1, 2}) = \phi_{3}(\beta_{(m_{1}, m_{2}+1, 0)})$, 

\item $\phi_{3}(\tau_{1, 2}\beta_{(m_{1}, m_{2}, m_{3})}\tau^{-1}_{1, 2}) = \phi_{3}(\beta_{(m_{1}, m_{2}+1, m_{3})})$,

\item $\phi_{3}(\beta^{-1}_{(m_{1}, m_{2}+1, 0)}\tau_{2, 1}\beta_{(m_{1}, m_{2}+1, 0)}\tau^{-1}_{2, 1}) = \phi_{3}(\tau_{2,1}\alpha_{(m_{1}, m_{2}+1, 0)}\alpha^{-1}_{(m_{1}, m_{2}, 0)}\tau^{-1}_{2,1})$,

\item $\phi_{3}(\beta^{-1}_{(m_{1}, m_{2}+1, m_{3})}\tau_{2, 1}\beta_{(m_{1}, m_{2}+1, m_{3})}\tau^{-1}_{2, 1}) = \phi_{3}(\tau_{2,1}\alpha_{(m_{1}, m_{2}+1, m_{3})}\alpha^{-1}_{(m_{1}, m_{2}, m_{3})}\tau^{-1}_{2,1})$,

\item $\phi_{3}(\tau_{3, 1}\alpha_{(m_{1}+1, 0, 0)}\tau^{-1}_{3, 1}) = \phi_{3}(\alpha_{(m_{1}, 0, 0)})$.
\end{enumerate}

Let $i \in \{1, 2, 3\}$. By equation (1) above and by Lemma \ref{leem7} (2) (b), we get \begin{equation*}
(\alpha^{-1}_{(m_{1}+1, 0, 0)}\tau_{3,2}\alpha_{(m_{1}+1, 0, 0)}\tau^{-1}_{3,2})_{i} = (\beta^{-1}_{(m_{1}+1, 0, 0)}\beta_{(m_{1}, 0, 0)})_{i}. 
\end{equation*}
Furthermore, by Lemma \ref{leem7} (2) (a), 
\begin{equation*}
\begin{split}
(\alpha^{-1}_{(m_{1}+1, 0, 0)}\tau_{3,2}\alpha_{(m_{1}+1, 0, 0)}\tau^{-1}_{3,2})_{i} &= (\alpha^{-1}_{(m_{1}+1, 0, 0)})_{i}(\tau_{3,2}\alpha_{(m_{1}+1, 0, 0)}\tau^{-1}_{3,2})_{i},\\
(\beta^{-1}_{(m_{1}+1, 0, 0)}\beta_{(m_{1}, 0, 0)})_{i} &= (\beta^{-1}_{(m_{1}+1, 0, 0)})_{i}(\beta_{(m_{1}, 0, 0)})_{i},
\end{split}
\end{equation*}
and by Lemma \ref{leem7} (1) (b), 
\begin{equation*}
(\tau_{3,2}\alpha_{(m_{1}+1, 0, 0)}\tau^{-1}_{3,2})_{i} = \tau_{3,2}(\alpha_{(m_{1}+1, 0, 0)})_{i}\tau^{-1}_{3,2}.
\end{equation*}
Thus, by the above equations,
\begin{equation*}
(\alpha^{-1}_{(m_{1}+1, 0, 0)})_{i}\tau_{3,2}(\alpha_{(m_{1}+1, 0, 0)})_{i}\tau^{-1}_{3,2} = (\beta^{-1}_{(m_{1}+1, 0, 0)})_{i}(\beta_{(m_{1}, 0, 0)})_{i}.
\end{equation*}
But by our notation, $(\alpha_{(m_{1}+1, 0, 0)})_{i} = \alpha_{i, (m_{1}+1, 0, 0)}$, $(\beta_{(m_{1}, 0, 0)})_{i} = \beta_{i, (m_{1}, 0, 0)}$ and $(\beta_{(m_{1}+1, 0, 0)})_{i} = \beta_{i, (m_{1}+1, 0, 0)}$ and thus, we obtain Lemma \ref{leem8} (1). 

By the above equations (2)-(6) and by analogous arguments, we get Lemma \ref{leem8} (1)-(6). 

Furthermore, by \cite[proof of Proposition 6]{kof1},
\begin{equation*}
\phi_{3}(\beta^{-1}_{(m_{1}, m_{3}, 0)}\tau_{1, 3}\beta_{(m_{1}, m_{3}, 0)} \tau^{-1}_{1, 3}) = \phi_{3}(\sigma_{2, 3}\beta^{-1}_{(m_{1}, 0, m_{3}+1)}\beta_{(m_{1}, 0, m_{3})}\sigma^{-1}_{2, 3}).
\end{equation*}
Let $\psi_{1} = \sigma_{2, 3}\beta^{-1}_{(m_{1}, 0, m_{3}+1)}\sigma^{-1}_{2, 3}$ and $\psi_{2} = \sigma_{2, 3}\beta_{(m_{1}, 0, m_{3})}\sigma^{-1}_{2, 3}$. It is easily verified that 
\begin{equation*}
\psi_{1}(x_{1}) = x_{1}[B^{x_{1}^{m_{1}}x_{2}^{m_{3}+1}}, C], \psi_{1}(x_{j}) = x_{j}, j = 2,3,
\end{equation*}
and 
\begin{equation*}
\psi_{2}(x_{1}) = x_{1}[B^{x_{1}^{m_{1}}x_{2}^{m_{3}}}, C]^{-1}, \psi_{2}(x_{j}) = x_{j}, j = 2,3.
\end{equation*}
Hence, $\phi_{3}(\beta^{-1}_{(m_{1}, m_{3}, 0)}\tau_{1, 3}\beta_{(m_{1}, m_{3}, 0)} \tau^{-1}_{1, 3}) = \phi_{3}(\psi_{1}\psi_{2})$. 
Let $k \in \{1, 2, 3\}$. By arguments similar to the above, we get $\beta^{-1}_{k, (m_{1}, m_{3}, 0)}\tau_{1, 3}\beta_{k, (m_{1}, m_{3}, 0)} \tau^{-1}_{1, 3} = (\psi_{1})_{k}(\psi_{2})_{k}$. But it is easily verified that
\begin{equation*}
\begin{split}
(\psi_{1})_{k} &= \sigma_{2, 3}\beta^{-1}_{(23)(k), (m_{1}, 0, m_{3}+1)}\sigma^{-1}_{2, 3}, \\
(\psi_{2})_{k} &= \sigma_{2, 3}\beta_{(23)(k), (m_{1}, 0, m_{3})}\sigma^{-1}_{2, 3},\end{split}
\end{equation*}
where $(23)(k)$ is the image of $k$ via the permutation $(23)$ of the symmetric group $S_{3}$. Hence, we may deduce Lemma \ref{leem8} (7).
\end{proof}

\section{Further analysis of $H_{3}$}\label{sec5}

\subsection{A Reduction of the Generating Set of $H_{3}$}

The relations in Lemma \ref{leem8} allow us to successively eliminate two of the parameters in $\alpha_{i, (m_{1}, m_{2}, m_{3})}$ ($i = 1, 2, 3$) and reduce these automorphisms to a one-parameter form.

\begin{lemma}\label{leem9}
Let $B_{3} = \langle \alpha_{1, (m, 0, 0)}, \alpha_{2, (m, 0, 0)}, \alpha_{3, (m, 0, 0)} : m \geq 0\rangle^{T_{3}}$. Then:
\begin{enumerate}
\item $\beta_{i, (m_{1}, m_{2}, m_{3})} \in B_{3}$ for all $i \in \{1, 2, 3\}$ and $m_{1}, m_{2}, m_{3} \geq 0$. 

\item $\alpha_{i, (m_{1}, m_{2}, m_{3})} \in B_{3}$ for all $i \in \{1, 2, 3\}$ and $m_{1}, m_{2}, m_{3} \geq 0$.

\item $H_{3} = B_{3}$.
\end{enumerate}
\end{lemma}

\begin{proof}
\begin{enumerate}
\item We split the proof into several steps. Let $i \in \{1, 2, 3\}$.

\begin{enumerate}
\item By definition, $\alpha_{i, (m_{1}, 0, 0)} \in B_{3}$ for all $m_{1} \geq 0$. Thus, using Lemma \ref{leem8} (1), we obtain $\beta^{-1}_{i, (m_{1}+1, 0, 0)}\beta_{i, (m_{1}, 0, 0)} \in B_{3}$ for all $m_{1} \geq 0$. Furthermore, by Lemma \ref{leem6} (1) (a)-(c), we deduce that $\beta_{i, (0, 0, 0)} \in B_{3}$. Therefore, by induction on $m_{1}$, we get $\beta_{i, (m_{1}, 0, 0)} \in B_{3}$ for all $m_{1} \geq 0$. 

\item By the above step, $\beta_{i, (m_{1}, 0, 0)} \in B_{3}$ for all $m_{1} \geq 0$, and by Lemma \ref{leem8} (2), $\tau_{1, 2}\beta_{i, (m_{1}, m_{2}, 0)}\tau^{-1}_{1, 2} = \beta_{i, (m_{1}, m_{2}+1, 0)}$ for all $m_{1}, m_{2} \geq 0$. Thus, by induction on $m_{2}$ we get $\beta_{i, (m_{1}, m_{2}, 0)} \in B_{3}$ for all $m_{1}, m_{2} \geq 0$. 

\item Since, by the above step, $\beta_{i, (m_{1}, m_{2}, 0)} \in B_{3}$ for all $m_{1}, m_{2} \geq 0$, it follows from Lemma \ref{leem8} (7) that $\beta^{-1}_{i, (m_{1}, 0, m_{3}+1)}\beta_{i, (m_{1}, 0, m_{3})} \in B_{3}$ for all $m_{1}, m_{3} \geq 0$. Hence, by induction on $m_{3}$, we get $\beta_{i, (m_{1}, 0, m_{3})} \in B_{3}$ for all $m_{1}, m_{3} \geq 0$. 

\item By the above step, $\beta_{i, (m_{1}, 0 , m_{3})} \in B_{3}$ for all $m_{1}, m_{3} \geq 0$, and by Lemma \ref{leem8} (3), $\tau_{1, 2}\beta_{i, (m_{1}, m_{2}, m_{3})}\tau^{-1}_{1, 2} = \beta_{i, (m_{1}, m_{2}+1, m_{3})}$ for all $m_{1}, m_{2}, m_{3} \geq 0$. Therefore, by induction on $m_{2}$, $\beta_{i, (m_{1}, m_{2} , m_{3})} \in B_{3}$ for all $m_{1}, m_{2}, m_{3} \geq 0$.
\end{enumerate}

\item We argue similarly to the above. Let $i \in \{1, 2, 3\}$.

\begin{enumerate}
\item By part (1), $\beta_{i, (m_{1}, m_{2}, 0)} \in B_{3}$ for all $m_{1}, m_{2} \geq 0$. Thus, using Lemma \ref{leem8} (4), we obtain $\alpha_{i, (m_{1}, m_{2}+1, 0)}\alpha^{-1}_{i, (m_{1}, m_{2}, 0)} \in B_{3}$ for all $m_{1}, m_{2} \geq 0$. Since, by definition, $\alpha_{i, (m_{1}, 0, 0)} \in B_{3}$ for all $m_{1} \geq 0$, by induction on $m_{2}$ we get $\alpha_{i, (m_{1}, m_{2}, 0)} \in B_{3}$ for all $m_{1}, m_{2} \geq 0$. 

\item By the above step, $\alpha_{i, (m_{1}, m_{2}, 0)} \in B_{3}$ for all $m_{1}, m_{2} \geq 0$. Therefore, it follows from Lemma \ref{leem6} (2), applied for $m_{2} = 0$, that $\alpha_{i, (m_{1}, 0, m_{3}+1)} \in B_{3}$ for all $m_{1}, m_{3} \geq 0$. Since, by definition, $\alpha_{i, (m_{1}, 0, 0)} \in B_{3}$ for all $m_{1} \geq 0$, we get $\alpha_{i, (m_{1}, 0, m_{3})} \in B_{3}$ for all $m_{1}, m_{3} \geq 0$. 

\item By part (1), $\beta_{i, (m_{1}, m_{2}, m_{3})} \in B_{3}$, and by Lemma \ref{leem8} (5), 
\begin{equation*}
\beta^{-1}_{i, (m_{1}, m_{2}+1, m_{3})}\tau_{2, 1}\beta_{i, (m_{1}, m_{2}+1, m_{3})}\tau^{-1}_{2, 1} = \tau_{2,1}\alpha_{i, (m_{1}, m_{2}+1, m_{3})}\alpha^{-1}_{i, (m_{1}, m_{2}, m_{3})}\tau^{-1}_{2,1}
\end{equation*}
for all $m_{1}, m_{2}, m_{3} \geq 0$. Thus, $\alpha_{i, (m_{1}, m_{2}+1, m_{3})}\alpha^{-1}_{i, (m_{1}, m_{2}, m_{3})} \in B_{3}$ for all $m_{1}, m_{2}, m_{3} \geq 0$. Since, by the above step, $\alpha_{i, (m_{1}, 0, m_{3})} \in B_{3}$ for all $m_{1}, m_{3} \geq 0$, by induction on $m_{2}$ we get $\alpha_{i, (m_{1}, m_{2}, m_{3})} \in B_{3}$ for all $m_{1}, m_{2}, m_{3} \geq 0$.
\end{enumerate}

\item By Lemma \ref{leem5}, $H_{3}= \langle \alpha_{1, (m_{1}, m_{2}, m_{3})}, \alpha_{2, (m_{1}, m_{2}, m_{3})}, \alpha_{3, (m_{1}, m_{2}, m_{3})}: m_{1}, m_{2}, m_{3} \geq 0\rangle^{T_{3}}$ and by part (2), $\alpha_{i, (m_{1}, m_{2}, m_{3})} \in B_{3}$ for all $i \in \{1, 2, 3\}$ and $m_{1}, m_{2}, m_{3} \geq 0$. Hence, it follows that $H_{3} = B_{3}$.
\end{enumerate}
\end{proof}

\begin{proposition}\label{propo5}
$H_{3} = \langle \alpha_{1, (0, 0, 0)}, \alpha_{2, (0, 0, 0)}\rangle^{T_{3}}$, 
where $\alpha_{1, (0, 0, 0)}$ and $\alpha_{2, (0, 0, 0)}$ are the automorphisms of $G_{3}$ defined by
\begin{align*}
\alpha_{1, (0, 0, 0)}(x_{1}) &= x_{1}[[x_{1}, x_{2}], [x_{1}, x_{3}], x_{1}], \alpha_{1, (0, 0, 0)}(x_{i})= x_{i}, i = 2, 3,\\
\alpha_{2, (0, 0, 0)}(x_{1}) &= x_{1}[[x_{1}, x_{2}], [x_{1}, x_{3}], x_{2}], \alpha_{2, (0, 0, 0)}(x_{i})= x_{i}, i = 2, 3.
\end{align*}
\end{proposition}

\begin{proof}
By Lemma \ref{leem9} (3), $H_{3} =\langle \alpha_{1, (m, 0, 0)}, \alpha_{2, (m, 0, 0)}, \alpha_{3, (m, 0, 0)} : m \geq 0\rangle^{T_{3}}$ and hence it suffices to show that $\alpha_{i, (m, 0, 0)} \in  \langle \alpha_{1, (0, 0, 0)}, \alpha_{2, (0, 0, 0)}\rangle^{T_{3}}$ for all $i \in \{1, 2, 3\}$ and $m \geq 0$. Let $i \in \{1, 2 ,3\}$. Since, by Lemma \ref{leem8} (6), $\tau_{3, 1}\alpha_{i, (m+1, 0, 0)}\tau^{-1}_{3, 1} = \alpha_{i, (m, 0, 0)}$, by induction on $m$ we get $\alpha_{i, (m, 0, 0)} \in \langle \alpha_{1, (0, 0, 0)}, \alpha_{2, (0, 0, 0)}, \alpha_{3, (0, 0, 0)} \rangle^{T_{3}}$. Therefore, $H_{3} \subseteq \langle \alpha_{1, (0, 0, 0)}, \alpha_{2, (0, 0, 0)}, \alpha_{3, (0, 0, 0)} \rangle^{T_{3}}$. Furthermore, by Lemma \ref{leem6} (3), $\sigma^{-1}_{2, 3}\alpha^{-1}_{2, (0, 0, 0)}\sigma_{2, 3} = \alpha_{3, (0, 0, 0)}$ and thus, the result follows.
\end{proof}

\section{Proof of Theorem \ref{the}}\label{lastsec}

\begin{lemma}\label{leem10}
Let $w$ be the automorphism of $G_{3}$ defined by $w(x_{1}) = x_{1}[[x_{1},x_{2}],[x_{1},x_{3}]]$ and $w(x_{i}) = x_{i}$ for $i=2,3$. Then, 
$\langle T_{3}, w, \alpha_{2, (0, 0, 0)}\rangle = \langle T_{3}, w \rangle$.
\end{lemma}

\begin{proof}
It suffices to show that $\alpha_{2, (0, 0, 0)} \in \langle T_{3}, w \rangle$. We claim that $wk^{-1}_{2}\pi_{3,2}w\pi^{-1}_{3,2}k_{2} = \alpha_{2, (0, 0, 0)}$. Observe that $wk^{-1}_{2}\pi_{3,2}w\pi^{-1}_{3,2}k_{2}(x_{j}) = x_{j} = \alpha_{2, (0, 0, 0)}(x_{j})$ for $j = 2, 3$. Thus, to prove our claim, it is enough to verify that $wk^{-1}_{2}\pi_{3,2}w\pi^{-1}_{3,2}k_{2}(x_{1}) = \alpha_{2, (0, 0, 0)}(x_{1})$. By direct calculations, we get
\begin{equation*}
k^{-1}_{2}\pi_{3,2}w\pi^{-1}_{3,2}k_{2}(x_{1}) = x_{1}[[x_{1},x^{-1}_{2}],[x_{1},x_{3}x^{-1}_{2}]].
\end{equation*}
By the group commutator identity $[a, bc] = [a, c][a, b]^{c}$ and by Lemma \ref{leem2} (3), 
\begin{equation*}
[[x_{1}, x^{-1}_{2}],[x_{1},x_{3}x^{-1}_{2}]] = [[x_{1}, x^{-1}_{2}],[x_{1}, x^{-1}_{2}][x_{1}, x_{3}]^{x^{-1}_{2}}] = [[x_{1}, x^{-1}_{2}],[x_{1}, x_{3}]^{x^{-1}_{2}}].
\end{equation*}
Hence, by our notation, by the group commutator identity $[a, b^{-1}] = ([a, b]^{-1})^{b^{-1}}$, by Lemma \ref{leem2} (1) and by Lemma \ref{leem2} (8), we get
\begin{equation*}
\begin{split}
[[x_{1}, x^{-1}_{2}],[x_{1}, x_{3}x^{-1}_{2}]] &= [[x_{1}, x^{-1}_{2}],[x_{1}, x_{3}]^{x^{-1}_{2}}] = [(A^{-1})^{x^{-1}_{2}}, B^{x^{-1}_{2}}] = [A^{x^{-1}_{2}}, B^{x^{-1}_{2}}]^{-1}\\&= ([A, B][A, B, x^{-1}_{2}])^{-1} = [A, B]^{-1}[A, B, x^{-1}_{2}]^{-1}\\ &= [A, B]^{-1}[A, B, x_{2}].
\end{split}
\end{equation*}
Hence, $k^{-1}_{2}\pi_{3,2}w\pi^{-1}_{3,2}k_{2}(x_{1}) = x_{1}[A, B]^{-1}[A, B, x_{2}]$. Therefore, it follows directly that 
\begin{equation*}
wk^{-1}_{2}\pi_{3,2}w\pi^{-1}_{3,2}k_{2}(x_{1}) = x_{1}[A, B, x_{2}] = \alpha_{2, (0, 0, 0)}(x_{1}). 
\end{equation*}
Hence, our claim is true and thus, $\alpha_{2, (0, 0, 0)} \in \langle T_{3}, w \rangle$.
\end{proof}

\begin{proof}[Proof of Theorem \ref{the}]
Since, by Proposition \ref{propo5}, $H_{3} = \langle \alpha_{1, (0, 0, 0)}, \alpha_{2, (0, 0, 0)}\rangle^{T_{3}}$, we get $\langle T_{3}, \mu, w \rangle H_{3} = \langle T_{3}, \mu, w, \alpha_{1, (0, 0, 0)}, \alpha_{2, (0, 0, 0)}\rangle$. Since, by Proposition \ref{propo4}, $\langle T_{3}, \mu, w\rangle H_{3}$ is dense in ${\rm Aut}(G_{3})$, it follows that $\langle T_{3}, \mu, w, \alpha_{1, (0, 0, 0)}, \alpha_{2, (0, 0, 0)}\rangle$ is dense in ${\rm Aut}(G_{3})$. Furthermore, by our notation, $\alpha_{1, (0, 0, 0)} = \alpha$, and by Lemma \ref{leem10}, $\langle T_{3}, w, \alpha_{2, (0, 0, 0)}\rangle = \langle T_{3}, w \rangle$. Therefore, $\langle T_{3}, \mu, w, \alpha_{1, (0, 0, 0)}, \alpha_{2, (0, 0, 0)}\rangle = \langle T_{3}, \mu, w, \alpha\rangle$, and thus, the subgroup $\langle T_{3}, \mu, w, \alpha\rangle$ is dense in ${\rm Aut}(G_{3})$. Since, by Lemma \ref{leem1} (2)-(3), there is a group epimorphism from ${\rm Aut}(G_{3})$ onto ${\rm Aut}(M_{3})$ and, by \cite{bamu}, ${\rm Aut}(M_{3})$ is not finitely generated, it follows that $\langle T_{3}, w, \mu, \alpha\rangle$ is a proper subgroup of ${\rm Aut}(G_{3})$.
\end{proof}

\end{document}